\documentclass[final]{siamltex}

\usepackage{amsfonts}
\usepackage{graphicx}
\usepackage{amsmath}
\usepackage{graphicx}
\usepackage{amssymb}
\usepackage[colorlinks,linkcolor={blue}]{hyperref}

\newtheorem{remark}[theorem]{Remark}

\title{Lyapunov Orbits at $L_2$ and Transversal Intersections of Invariant Manifolds in the Jupiter-Sun Planar Restricted Circular Three Body Problem}

\author{Maciej J. Capi\'nski \thanks{The work was initiated during a visit of the author to University of Texas at Austin, sponsored by the Ko\'sciuszko Foundation. The work has been supported by the Polish State Ministry of Science and Information Technology grant N201 543238.}
        }

\begin{document}

\maketitle

\begin{abstract}
We present a computer assisted proof of existence of a family of Lyapunov orbits which stretches from $L_2$ up to half the distance to the smaller primary in the Jupiter-Sun planar restricted circular three body problem. We then focus on a small family of Lyapunov orbits with energies close to comet Oterma and show that their associated invariant manifolds intersect transversally. Our computer assisted proof provides explicit bounds on the location and on the angle of intersection.
\end{abstract}



\pagestyle{myheadings}
\thispagestyle{plain}
\markboth{MACIEJ J. CAPI\'NSKI}{INTERSECTIONS OF INVARIANT MANIFOLDS IN THE PRC3BP}

%TCIDATA{OutputFilter=latex2.dll}
%TCIDATA{Version=5.00.0.2606}
%TCIDATA{LaTeXparent=0,0,online-edit.tex}
\section{Introduction\label{sec:introduction}}
The Planar Restricted Circular Three Body Problem (PRC-3BP) has been extensively studied throughout literature. The model has applications in space mission design \cite{GJSM, GKLMMR}, explains symbolic dynamics phenomena observed in trajectories of comets \cite{KLMR} and can be used for study of diffusion estimates \cite{JorbaSimo94, JorbaVillanueva98}. All the above mentioned are associated with dynamics along invariant manifolds of the system. In this paper we discuss how existence of such manifolds can be proved within explicit bounds using rigorous-computer-assisted techniques.

We focus on dynamics associated with the fixed point $L_2$, its associated center manifold and stable/unstable manifolds. The problem has been studied by Llibre, Martinez and Simo \cite{Simo} where under appropriate conditions on parameters of the system existence and intersections of such manifolds has been proved analytically. In the work of Koon, Lo, Marsden and Ross \cite{KLMR} such invariant manifolds and their associated symbolic dynamics have been used to numerically explain a peculiar trajectory of the comet Oterma in the vicinity of Jupiter. Such symbolic dynamics has later been proved using rigorous-computer-assisted computations by Wilczak and Zgliczy\'nski \cite{WZ1, WZ2}. The work presented in this paper can be viewed as an extension of last-mentioned. Results \cite{WZ1, WZ2} were obtained using purely topological arguments. They focus on homoclinic and heteroclinic tangle between periodic orbits, without the detection of the manifolds themselves or angles of their intersections. Here we address these issues.

In this paper we shall first present a method for detecting of families of Lyapunov orbits in the PRC3BP. It is designed as a tool for rigorous-computer-assisted proofs. We apply the method to obtain a family that spans up to half a distance between the fixed point $L_2$ and the smaller primary in the Jupiter-Sun system. This is our first main result, which is stated in Theorem \ref{th:main-1}. The method is based on a combination of interval Newton method and implicit function theorem. 

We then consider a small family of Lyapunov orbits with energies close to the energy of comet Oterma. We prove that the family is normally hyperbolic, and give a tool for obtaining rigorous bounds for its unstable and stable fibers. The tool is based on a topological approach combined with a parameterization method. We then show how fibers can be propagated to prove transversal intersections between stable and unstable manifolds of Lyapunov orbits. We investigate an intersection associated with manifolds which span from the Lyapunov orbit and circle around the larger primary. We obtain explicit bounds on the location of intersection and also on its angle. This is the second main result of the paper, which is stated in Theorem \ref{th:main}.

Both methods which we propose are tailor made for the PRC3BP. We make use of the preservation of energy and reversibility of the system. Thanks to this our rigorous bounds for the investigated manifolds are quite sharp. 

For our method we also develop a more general tool which can be applied for the detection of  unstable/stable manifolds of saddle - center fixed points. It is a generalization of the wok of Zgliczy\'nski \cite{Z1}. This is the subject of section \ref{sec:unstable-topological}.

The paper is organized as follows. Section \ref{sec:preliminaries} includes preliminaries which give an introduction to the PRC3BP, the interval Newton method, and introduce some notations. In section \ref{sec:Lap-orb-man} we present a method for detection of families of Lyapunov orbits and apply it to the Jupiter-Sun system. In section \ref{sec:outline} we outline the results for the intersections of invariant manifolds, which are then proved throughout the remainder of the paper. In section \ref{sec:hyp-energy} we show how to prove that Lyapunov orbits are hyperbolic and foliated by energy. In section \ref{sec:unstable-topological} we give a topological tool for detection of unstable manifolds of saddle-center fixed points. The method is then combined with parametrization method in section \ref{sec:parameterization-fibers} to obtain rigorous bounds on the intersections of invariant manifolds. Sections \ref{sec:closing-remarks}, \ref{sec:ackn} and \ref{sec:appendix} contain respectively closing remarks, acknowledgements and the appendix.

\section{Preliminaries\label{sec:preliminaries}}

\subsection{The Planar Restricted Circular Three Body Problem}

In the Planar restricted circular three body problem (PRC3BP) we consider the
motion of a small massless particle under the gravitational pull of two larger
bodies (which we shall refer to as primaries) of mass $\mu$ and $1-\mu$. The primaries move around the origin on circular orbits of period $2\pi$ on the same plane
as the massless body. In this paper we shall consider the mass parameter
$\mu=0.0009537$, which corresponds to the rescaled mass of Jupiter in the
Jupiter-Sun system. 

The Hamiltonian of the problem is given by \cite{AM}%
\begin{equation}
H(q,p,t)=\frac{p_{1}^{2}+p_{2}^{2}}{2}-\frac{1-\mu}{r_{1}(t)}-\frac{\mu}%
{r_{2}(t)}, \notag%\label{eq:H-for-circular-problem}%
\end{equation}
where $\left(  p,q\right)  =\left(  q_{1},q_{2},p_{1},p_{2}\right)  $ are the
coordinates of the massless particle and $r_{1}(t)$ and $r_{2}(t)$ are the
distances from the masses $1-\mu$ and $\mu$ respectively. 

After introducing a
new coordinates system $(x,y,p_{x},p_{y})$%
\begin{equation}%
\begin{array}
[c]{ll}%
x=q_{1}\cos t+q_{2}\sin t, & \quad p_{x}=p_{1}\cos t+p_{2}\sin t,\\
y=-q_{1}\sin t+q_{2}\cos t, & \quad p_{y}=-p_{1}\sin t+p_{2}\cos t,
\end{array}
\label{eq:x,y-coordinates}%
\end{equation}
which rotates together with the primaries, the primaries become
motionless (see Figure \ref{fig:forbidden-region}) and one obtains \cite{AM} an autonomous Hamiltonian%
\begin{equation}
H(x,y,p_{x},p_{y})=\frac{(p_{x}+y)^{2}+(p_{y}-x)^{2}}{2}-\Omega(x,y),
\label{eq:H-PRC3BP}%
\end{equation}
where%
\begin{align}
\Omega(x,y)  &  =\frac{x^{2}+y^{2}}{2}+\frac{1-\mu}{r_{1}}+\frac{\mu}{r_{2}%
},\nonumber\\
r_{1}  &  =\sqrt{(x-\mu)^{2}+y^{2}},\quad r_{2}=\sqrt{(x+1-\mu)^{2}+y^{2}}.
\notag %\label{eq:r1-r2}%
\end{align}

The motion of the particle is given by 
\begin{equation}
\dot{q}=J\nabla H(q), \label{eq:PRC3BP}%
\end{equation}
where $q=(x,y,p_{x},p_{y})\in\mathbb{R}^{4}$, $J=\left(
\begin{array}
[c]{cc}%
0 & \text{id}\\
-\text{id} & 0
\end{array}
\right)  $ and id is a two dimensional identity matrix.

The movement of the flow (\ref{eq:PRC3BP}) is restricted to the hypersurfaces
determined by the energy level $h$,
\begin{equation}
M(h)=\{(x,y,p_{x},p_{y})\in\mathbb{R}^{4}|H(x,y,p_{x},p_{y})=h\}.\label{eq:M-energy}
\end{equation}
This means that movement in the $x,y$ coordinates is restricted to the so
called Hill's region defined by
\begin{equation}
R(h)=\{(x,y)\in\mathbb{R}^{2}|\Omega(x,y)\geq-h\}.\nonumber
\end{equation}
\begin{figure}[h]
\begin{center}
\includegraphics[height=2.4in]{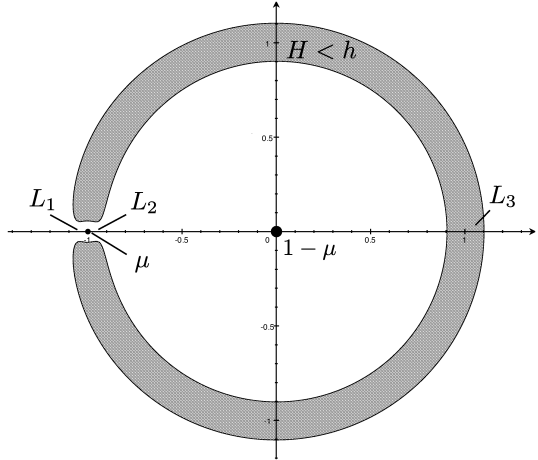}
\end{center}
\caption{The Hill's region for the energy level $h=1.515$ of comet Oterma in the Jupiter-Sun system.}
\label{fig:forbidden-region}
\end{figure}

The problem has three equilibrium points $L_{1},L_{2},L_{3}$ on the $x$-axes (see Figure \ref{fig:forbidden-region}).
We shall be interested in the dynamics associated with $L_{2}$, and with orbits of energies higher than that of $L_2$. The linearized
vector field at the point $L_{2}$ has two real and two purely imaginary
eigenvalues, thus by the Lyapunov theorem (see for example
\cite{Simo}) for energies $h$ larger and sufficiently close to $H(L_{2})$
there exists a family of periodic orbits parameterized by energy emanating
from the equilibrium point $L_{2}.$ 
Numerical evidence shows that this family extends up to and even beyond the smaller primary $\mu$ \cite{Broucke}. 

\begin{figure}[h]
\begin{center}
\includegraphics[height=1.8in]{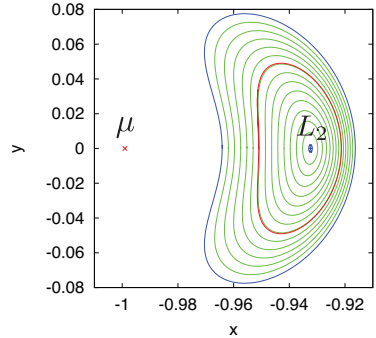}
\end{center}
\caption{Considered by us family of Lyapunov orbits in green (spanning between two orbits in blue), together with the Lyapunov orbit for the energy of the comet Oterma $h=1.\,\allowbreak515$ in red.}
\label{fig:LapOrbits}%
\end{figure}

The PRC3BP admits the following reversing symmetry 
\[
S(x,y,p_{x},p_{y})=(x,-y,-p_{x},p_{y}).
\] 
For the flow $\phi(t,q)$ of (\ref{eq:PRC3BP}) we have%
\begin{equation}
S(\phi(t,q))=\phi(-t,S(q)). \label{eq:sym-prop}%
\end{equation}
We will say that an orbit $q(t)$ is $S$-symmetric when
\begin{equation}
S(q(t))=q(-t). \label{eq:S-sym}%
\end{equation}

Each Lyapunov orbit is $S$-symmetric. It possesses a two dimensional stable manifold and a two dimensional unstable manifold. These manifolds lie on the same energy level as the orbit and their intersection, when restricted to the three dimensional constant energy manifold (\ref{eq:M-energy}), is transversal. These invariant manifolds are $S$-symmetric with respect to each other, meaning that the stable manifold is an image by $S$ of the unstable manifold (see Figure \ref{fig:plot3d} for the unstable manifold, and Figure \ref{fig:plot2d} for the intersection of manifolds). All these facts are well known and extensively studied numerically. 

Our aim in this paper will be firstly to provide a rigorous-computer-assisted proof of existence of the manifold of Lyapunov orbits over a large radius from $L_2$ (see Figure \ref{fig:LapOrbits}). Secondly, using rigorous-computer-assisted computations, we shall show that for orbits with energies close to the energy of comet Oterma $h=1.\,\allowbreak515$ their associated stable and unstable manifolds intersect transversally. Even though such intersections are well known from numerical investigation, to the best of our knowledge this is a first rigorous proof of their existence. 

\begin{figure}[h]
\begin{center}
\includegraphics[height=3.6in]{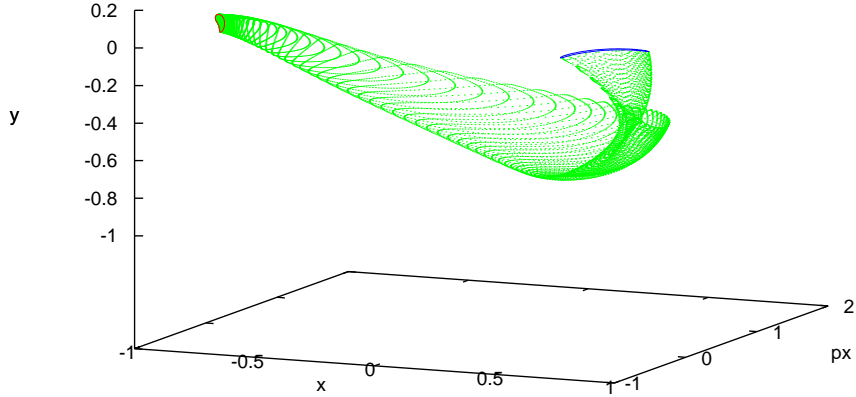}
\end{center}
\caption{The Lyapunov orbit in red, its unstable manifold in green, and the intersection of the unstable manifold with section $\{ y=0\}$ in blue, projected onto $x,y,p_x$ coordinates. The figure is for the energy of comet Oterma $h=1.515$ in the Jupiter-Sun system.}%
\label{fig:plot3d}%
\end{figure}

\begin{figure}[h]
\begin{center}
\includegraphics[height=1.9in]{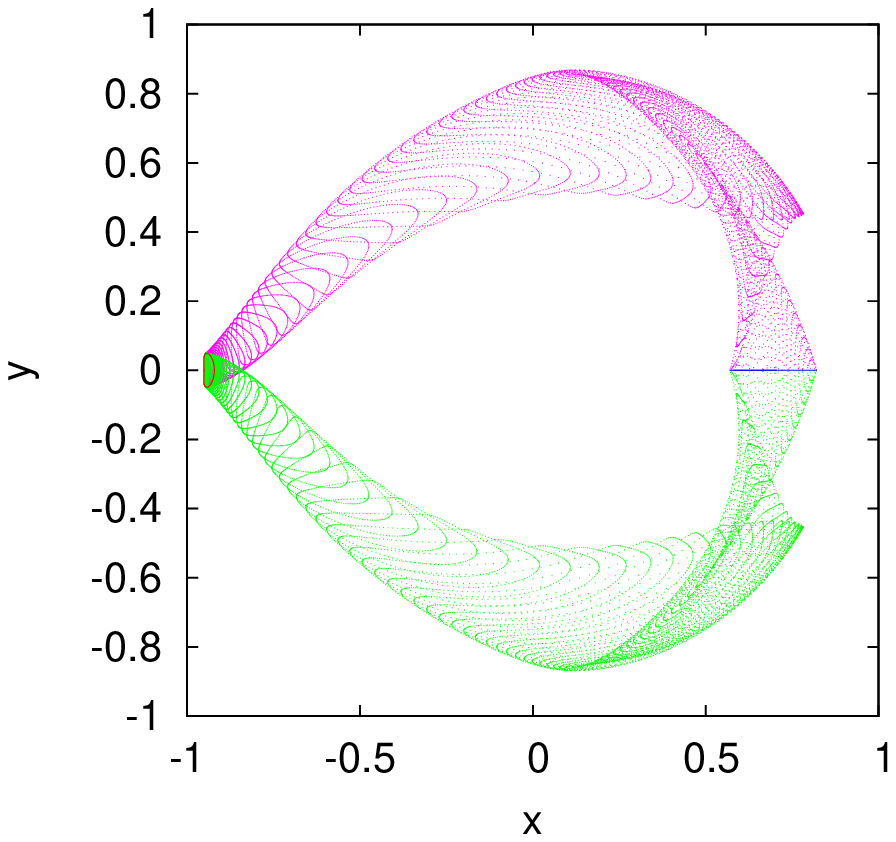}\includegraphics[height=1.9in]{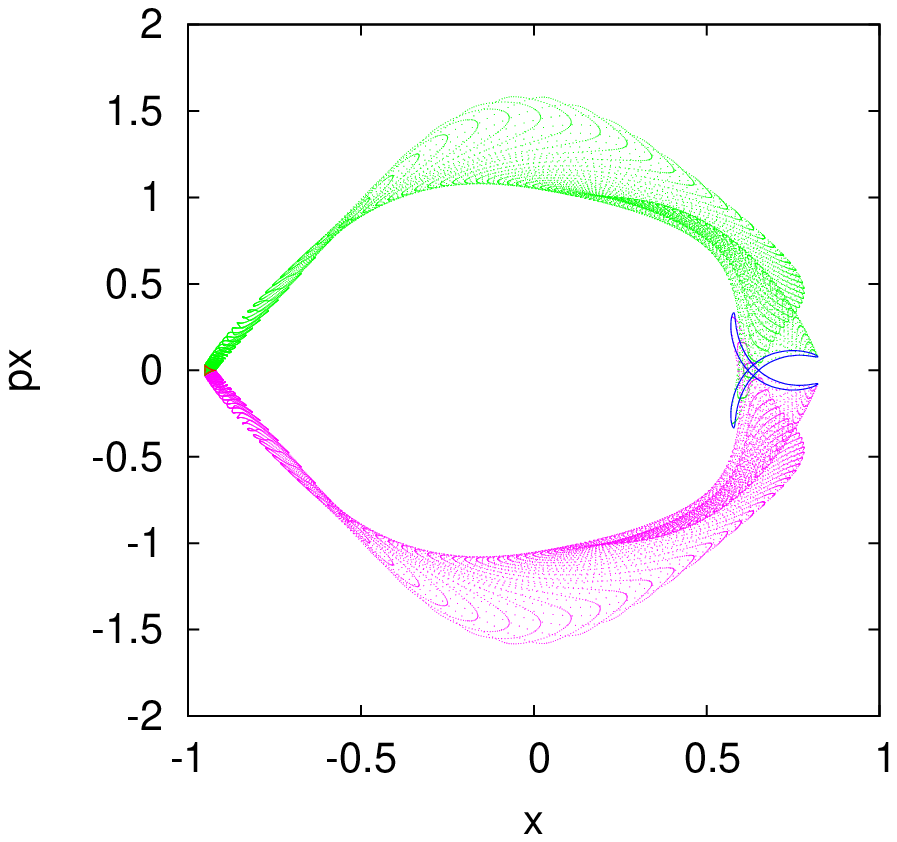}
\end{center}
\caption{The Lyapunov orbit in red, its unstable manifold in green, stable manifold in purple, and their intersections with section $\{ y=0\}$ in blue, projected onto $x,y$ coordinates (left) and $x,p_x$ coordinates (right). The figure is for the energy of comet Oterma $h=1.\,\allowbreak515$ in the Jupiter-Sun system.}%
\label{fig:plot2d}%
\end{figure}

\subsection{Interval Newton Method}

Let $X$ be a subset of $\mathbb{R}^n$. We shall denote by $[X]$ an interval enclosure of the set $X$, that is, a set
\[
[X]=\Pi_{i=1}^{n} [a_i,b_i]\subset \mathbb{R}^n,
\]
such that
\[
X\subset [X].
\]

Let $f:\mathbb{R}^n\to \mathbb{R}^n$ be a $C^1$ function and $U\subset \mathbb{R}^n$. We shall denote by $[Df(U)]$ the interval enclosure of a Jacobian matrix on the set $U$. This means that $[Df(U)]$ is an interval matrix defined as
\[
\lbrack Df(U)]=\left\{  A\in\mathbb{R}^{n\times n}|A_{ij}\in\left[  \inf_{x\in
U}\frac{df_{i}}{dx_{j}}(x),\sup_{x\in U}\frac{df_{i}}{dx_{j}}(x)\right]
\text{ for all }i,j=1,\ldots,n\text{ }\right\}  .
\]

\begin{theorem}
\label{th:interval-Newton}\cite{Al} (Interval Newton method) Let
$f:\mathbb{R}^{n}\rightarrow\mathbb{R}^{n}$ be a $C^{1}$ function and
$X=\Pi_{i=1}^{n}[a_{i},b_{i}]$ with $a_{i}<b_{i}$. If $[Df(X)]$ is invertible
and there exists an $x_{0}$ in $X$ such that%
\[
N(x_{0},X):=x_{0}-\left[  Df(X)\right]  ^{-1}f(x_{0})\subset X,
\]
then there exists a unique point $x^{\ast}\in X$ such that $f(x^{\ast})=0.$
\end{theorem}

\subsection{Notations}
Throughout the paper we shall use a notation $\phi(t,x)$ for the flow, and $\Phi_{T}(x)=\phi(T,x)$ for a time $T$ shift along trajectory map of (\ref{eq:PRC3BP}). For points $p=(x,y)$ we shall write $\pi_{x}p$ and $\pi_{y}p,$ to denote projections onto coordinates $x$ and $y$ respectively. We shall also use the following notation for a cartesian product of sets $\Pi_{i=1}^{n}U_{i}=U_{1}\times\ldots\times U_{n}$. For $A,B\subset\mathbb{R}^{n}$ we shall use a notation $A+B=\{a+b|a\in A,b\in
B\}$.

%\usepackage{pdfsync}
%\usepackage{hyperref}%
%\usepackage{pdfsync}
%\usepackage{hyperref}%
%\usepackage{pdfsync}
%\usepackage{hyperref}%

%TCIDATA{OutputFilter=latex2.dll}
%TCIDATA{Version=5.00.0.2606}
%TCIDATA{LaTeXparent=0,0,online-edit.tex}

\section{Existence of a Family of Lyapunov Orbits\label{sec:Lap-orb-man}}

In this section we shall present a method for proving existence of Lyapunov
orbits far away from $L_{2}$. The result is in the spirit of the method
applied by Wilczak and Zgliczy\'{n}ski in \cite{WZ1,WZ2} for a Lyapunov orbit
with energy $h=1.\,515$ of the comet Oterma. Our result differs from
\cite{WZ1,WZ2} by the fact that we obtain a smooth family of orbits over a
large set, whereas in \cite{WZ1,WZ2} a single orbit was proved.

We shall consider orbits starting from points of the form $(x,0,0,p_{y})$ with
$x$ inside an interval
\begin{align}
I_{x}  &  =[\underline{I_{x}},\overline{I_{x}}]:=\left[  \frac{1}{2}%
(-1+\mu-0.933),-0.933\right] \label{eq:Ix-interval}\\
&  \approx\lbrack-0.96602315,-0.933]\subset\mathbb{R}.\nonumber
\end{align}
Since $\pi_{x}L_{2}\approx-0.93237$ we see that $\underline{I_{x}}<\frac{1}%
{2}(-1+\mu-\pi_{x} L_{2})$, so the interval $I_{x}$ stretches from half the
distance between the smaller primary and $L_{2}$, almost up to $L_{2}$ (see
Figure \ref{fig:LapOrbits}, where the orbits are depicted in green, and
stretch between an inner and outer orbit depicted in blue).

Let us consider a section $\Sigma=\{y=0\}$ and a Poincar\'{e} map
$P:\Sigma\rightarrow\Sigma$ of (\ref{eq:PRC3BP}). We shall interpret the
Poincar\'{e} map as a function from $\mathbb{R}^{3}$ to $\mathbb{R}^{3}$ with
coordinates $x,p_{x},p_{y}$. If for a point $q=(x,0,p_{y})\in\Sigma$ we have
$\pi_{p_{x}}P(q)=0$, then by the symmetry property (\ref{eq:sym-prop}) the
point $q$ lies on a periodic orbit (the Poincar\'{e} map $P$ makes a half turn
along the orbit starting from $q$).

Let us introduce the following notation
\[
f:\mathbb{R}^{2}\rightarrow\mathbb{R},
\]%
\[
f(x,p_{y})=\pi_{p_{x}}P(x,0,p_{y}).
\]
To find a periodic orbit for some fixed $x$ it is sufficient to find a zero of
a function
\[
g_{x}(p_{y}):=f(x,p_{y}).
\]
Let $DP=\left(  dP_{i\,j}\right)  _{i,j=1,2,3}$ be the derivative of the map
$P,$ with indexes $1,2,3$ corresponding to coordinates $x,p_{x},p_{y}$
respectively. \begin{figure}[h]
\begin{center}
\includegraphics[height=1.5in]{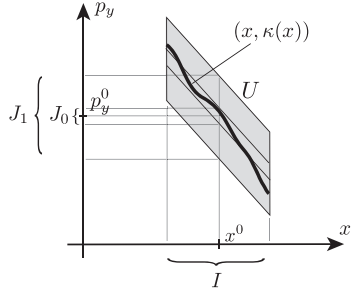}
\end{center}
\caption{The bound for a curve of points $q(x)=(x,0,0,\kappa(x))$ on Lyapunov
orbits.}%
\label{fig:U-bound}%
\end{figure}

\begin{lemma}
\label{lem:Newton-per-orb}Let $I$ and $J_{i}$ for $i=0,1$ be closed intervals
such that $J_{0},J_{1}$ have the same center point $p_{y}^{0}$ and
$J_{0}\subset J_{1}$. Let $x^{0}$ be the center point of $I$. Let
$a\in\mathbb{R}$ and $U_{0},U\subset\Sigma=\mathbb{R}^{3}$ be sets defined as
(see Figure \ref{fig:U-bound})
\begin{align}
U_{0}  &  =\{x^{0}\}\times\{0\}\times J_{0},\nonumber\\
U  &  =\left\{  (x,0,p_{y})|x\in I,p_{y}=a\left(  x-x^{0}\right)  +\iota
,\iota\in J_{1}\right\}  . \label{eq:set-U}%
\end{align}
If
\begin{equation}
N:=p_{y}^{0}-\left[  \frac{\pi_{p_{x}}P(x^{0},0,p_{y}^{0})}{dP(U_{0})_{2\,3}%
}\right]  \subset J_{0}, \label{eq:Newton-Lap}%
\end{equation}
and%
\begin{equation}
\left\vert \alpha-a\right\vert <\frac{1}{|I|}\left(  |J_{1}|-|J_{0}|\right)
\quad\text{for all }\alpha\in\left[  \underline{\alpha},\overline{\alpha
}\right]  :=\left[  -\frac{dP(U)_{2\,1}}{dP(U)_{2\,3}}\right]  ,
\label{eq:slope-assumption}%
\end{equation}
then there exists a smooth function $\kappa:I\rightarrow\mathbb{R}$ such that
for any $x\in I$ a point $q(x)=(x,0,0,\kappa(x))$ lies on an $S$-symmetric
periodic orbit of (\ref{eq:PRC3BP}). Moreover, $\kappa^{\prime}(x)\in\left[
\underline{\alpha},\overline{\alpha}\right]  $ and $q(x)\in U$ for all $x\in
I$.
\end{lemma}

\begin{proof}
Existence of a unique point $\kappa(x_{0})\in J_{0}$ for which $g_{x_{0}%
}(\kappa(x_{0}))=0$ follows from (\ref{eq:Newton-Lap}), which implies%
\[
p_{y}^{0}-\left[  Dg_{x_{0}}(J_{0})\right]  ^{-1}g_{x_{0}}(p_{y}^{0})\subset
N\subset J_{0},
\]
combined with interval Newton method (Theorem \ref{th:interval-Newton}).

For (\ref{eq:slope-assumption}) to hold we need to have $0\notin dP(U)_{2\,3}%
$. For $(x,0,p_{y})\in U$ we have $\frac{\partial f}{\partial p_{y}}%
(x,p_{y})\in dP(U)_{2\,3}$ hence $\frac{\partial f}{\partial p_{y}}%
(x,p_{y})\neq0$. This means that we can apply the implicit function theorem to
obtain a curve $\kappa(x)$ for which $f(x,\kappa(x))=0$. We now need to make
sure that the curve $\kappa$ is defined on the entire interval $I.$ At each
point $x$ for which $(x,0,\kappa(x))\in U$ is defined, by the implicit
function theorem we know that%
\[
\kappa^{\prime}(x)=-\frac{\frac{\partial f}{\partial x}(x,\kappa(x))}%
{\frac{\partial f}{\partial p_{y}}(x,\kappa(x))}\in\left[  -\frac
{dP(U)_{2\,1}}{dP(U)_{2\,3}}\right]  .
\]
This, by assumption (\ref{eq:slope-assumption}), means that we can continue
the curve from $\kappa(x^{0})$ to the whole interval $I$ (see Figure
\ref{fig:U-bound}).
\end{proof}

\begin{figure}[h]
\begin{center}
\includegraphics[height=1.8in]{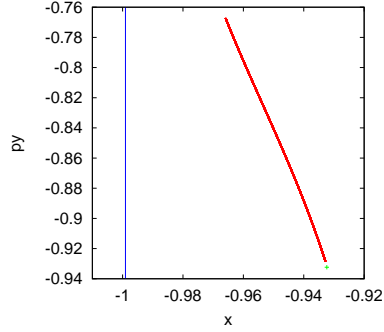}
\end{center}
\caption{Numerical plot of $\kappa(x)$, consisting of $15\,000$ points
$q_{i}^{0}$ on Lyapunov orbits (in red). The point $L_{2}$ is in green. The
blue line $x=-1+\mu$ gives an indication of the position of the smaller
primary along the $x$ coordinate.}%
\label{fig:Lyap-orb-numeric}%
\end{figure}To apply Lemma \ref{lem:Newton-per-orb} we first compute
numerically a sequence of points (see Figure \ref{fig:Lyap-orb-numeric})%
\begin{align*}
q_{i}^{0}  &  =(x_{i}^{0},0,0,p_{y,i}^{0})\quad\text{for}\quad i=0,\ldots
,15\,000,\\
x_{i}^{0}  &  =\underline{I_{x}}+\frac{i}{15000}\left(  \overline{I_{x}%
}-\underline{I_{x}}\right)  ,
\end{align*}
where $\underline{I_{x}}$, $\overline{I_{x}}$ are defined in
(\ref{eq:Ix-interval}). The $q_{i}^{0}$ are non-rigorously, numerically computed 
points on Lyapunov orbits. We then compute  (non-rigorously) a sequence of slopes (see Figure \ref{fig:slope}) 
\[
a_{i}\in\mathbb{R}\quad i=0,\ldots,15\,000,
\]
define%
\[
r=\frac{1}{15\,000}\frac{1}{2}(\overline{I_{x}}-\underline{I_{x}}%
)\approx10^{-6}\cdot1.1007716,
\]%
\begin{align*}
I_{i}  &  =x_{i}^{0}+\left[  -r,r\right]  ,\\
J_{0,i}  &  =p_{y,i}^{0}+10^{-13}\cdot\lbrack-1,1],\\
J_{1,i}  &  =p_{y,i}^{0}+10^{-8}\cdot\left[  -5,5\right]  ,
\end{align*}
and consider sets%
\begin{align*}
U_{0}  &  =\{x_{i}^{0}\}\times\{0\}\times J_{0,i},\\
U_{i}  &  =\left\{  (x,0,p_{y})|x\in I_{i},p_{y}=a_{i}\left(  x-x_{i}%
^{0}\right)  +\iota,\iota\in J_{1,i}\right\} .
\end{align*}

We apply Lemma \ref{lem:Newton-per-orb} repeatedly $15\,000$ times, and obtain
the following theorem.\begin{figure}[h]
\begin{center}
\includegraphics[height=1.8in]{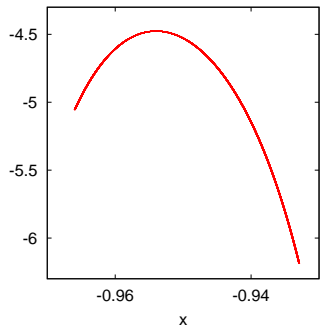}
\end{center}
\caption{Numerical plot of $\kappa^{\prime}(x)$, consisting of $15\,000$
points $a_{i}$}%
\label{fig:slope}%
\end{figure}

\begin{theorem}
[First main result]\label{th:main-1} Let $I_{x}$ be the interval from
(\ref{eq:Ix-interval}). Then there exists a curve $q(x)=(x,0,0,\kappa(x))$ of
points on Lyapunov orbits with $\kappa:I_{x}\rightarrow\mathbb{R}$, which lies
within a $5\cdot10^{-8}$ distance from the piecewise linear curve joining the
$15\,000$ points $q^{0}_{i}$ on Figure \ref{fig:Lyap-orb-numeric}.
\end{theorem}

The proof of Theorem \ref{th:main-1} took $5$ hours and $43$ minutes on a standard laptop.

\begin{remark}
Using above described method it is impossible to continue with the orbits to
$L_{2}$. At the fixed point one would need to apply alternative methods, such
as the method of majorants \cite{SM}, Lyapunov theorem by tracing the radius
of convergence of the normal form \cite{Moser}, or
topological-computer-assisted tools such as \cite{CR,CZ2}.
\end{remark}

%\usepackage{pdfsync}
%\usepackage{hyperref}%
%\usepackage{pdfsync}
%\usepackage{hyperref}%

%TCIDATA{OutputFilter=latex2.dll}
%TCIDATA{Version=5.00.0.2606}
%TCIDATA{LaTeXparent=0,0,online-edit.tex}

\section{Outline of Results for Intersections of Invariant Manifolds\label{sec:outline}}

In the reminder of the paper we shall focus our attention on orbits starting
from $q(x)=(x,0,0,\kappa(x))$ with $x\in I$ for
\begin{align}
I  &  = [\underline{I},\overline{I}]:= x^{0}+[-1,1]\cdot10^{-9}%
,\label{eq:x0-I}\\
x^{0}  &  =-0.9510055339445208. \notag
\end{align}
Such orbits have energy close to the energy of the comet Oterma $h=1.515$.

Let us introduce a notation $\Lambda$ for a family of Lyapunov orbits, which
start from $q(x)$ with $x\in I$%
\begin{equation}
\Lambda=\{\phi(t,q(x))|t\in\mathbb{R},q(x)=(x,0,0,\kappa(x)),x\in I\}.
\label{eq:Lambda-def}%
\end{equation}
For $x\in I,$ let $L(x)\subset\Lambda$ denote the Lyapunov orbit which starts
from $q(x).$

Throughout the reminder of the paper we shall prove the following theorem.

\begin{theorem}
[Second main result]\label{th:main} $\Lambda$ is a normally hyperbolic
invariant manifold with a boundary. Each orbit $L(x)\subset\Lambda$ possesses
a two dimensional stable manifold $W^{s}(L(x))$ and a two dimensional unstable
manifold $W^{u}(L(x)).$ The manifolds $W^{s}(L(x))$ and $W^{u}(L(x))$
intersect and the intersection, when restricted to the constant energy
manifold $M(H(L(x)))$, is transversal (see (\ref{eq:M-energy}) for definition
of $M$).
\end{theorem}

Numerical plots of the intersection of manifolds that we shall prove are given
in Figure \ref{fig:plot2d}.

Theorem \ref{th:main} will be proved with computer assistance. During the
proof we shall obtain rigorous bounds on the region and the angle at which the
manifolds intersect (see Figure \ref{fig:intersection}).

The size of interval $I$ (\ref{eq:x0-I}) is very small. When translated the
real life distance in the Jupiter-Sun system, its length is just slightly over
one and a half kilometer. This is practically a single point. We need to start
with such a small set to obtain our result. Thanks to this we obtain sharp
estimates on the intersection of $W^{s}(L(x))$, $W^{u}(L(x))$. To consider a
larger set of Lyapunov orbits one would need to iterate the procedure a number
of times. This can be done without any difficulty apart from necessary time
for computation. The proof of Theorem \ref{th:main} took $46$ minutes on a
standard laptop. Using clusters one could cover a larger interval $I$ in
reasonable time.

\section{Hyperbolicity of Lyapunov Orbits and Foliation by Energy\label{sec:hyp-energy}}

In this section we shall show that each orbit $L(x)\subset\Lambda$ lies on a
different energy level. We shall also show that each orbit $L(x)$ (when
considered on its constant energy manifold) is hyperbolic. In other words, we
shall show that $\Lambda$ is a normally hyperbolic manifold with a boundary.

We start with a simple remark.

\begin{remark}
\label{rem:energy-foliation}If for all $x\in I$ we have $\frac{d}%
{dx}H(q(x))\neq0,$ then Lyapunov orbits with different $x$ have different
energies. Note that the set $U$ and the bound on the derivative of
$\kappa^{\prime}(x)$ from Lemma \ref{lem:Newton-per-orb} can be used to
obtain
\[
\frac{d}{dx}H(q(x))\in\left[  \frac{\partial H}{\partial x}(U)+\frac{\partial
H}{\partial p_{y}}(U)\kappa^{\prime}(U)\right]  .
\]

\end{remark}

We shall now give a simple lemma which can be used to show that our Lyapunov
orbits are hyperbolic.

In what follows in this section, let $P:\Sigma\rightarrow\Sigma$ be a second
return Poincar\'{e} map for $\Sigma=\{y=0\}.$ This means that each point
$q(x)=(x,0,0,\kappa(x)),$ with $x\in I$, is a fixed point of $P.$ We shall
interpret the Poincar\'{e} map as a function from $\mathbb{R}^{3}$ to
$\mathbb{R}^{3}$ with coordinates $x,p_{x},p_{y}$.

\begin{lemma}
\label{lem:Normal-Hyperbolicity}Let $U$ be the set given by (\ref{eq:set-U})
in Lemma \ref{lem:Newton-per-orb}. Assume that for any $1\times2$ matrix $A$
\begin{equation}
A\in\left[  \left(  -\left(  \frac{\partial H}{\partial x}\right)
^{-1}\left(
\begin{array}
[c]{cc}%
\frac{\partial H}{\partial p_{x}} & \frac{\partial H}{\partial p_{y}}%
\end{array}
\right)  \right)  (U)\right]  \label{eq:A-matrix-def}%
\end{equation}
and any $2\times2$ matrix $B$%
\begin{equation}
B\in\left[  \left(  \left(
\begin{array}
[c]{c}%
dP_{2\,1}\\
dP_{3\,1}%
\end{array}
\right)  A+\left(
\begin{array}
[c]{cc}%
dP_{2\,2} & dP_{2\,3}\\
dP_{3\,2} & dP_{3\,3}%
\end{array}
\right)  \right)  \left(  U\right)  \right]  \label{eq:matrix-B}%
\end{equation}
the spectrum of $B$ consists of two real eigenvalues $\lambda_{1},\lambda_{2}$
satisfying $\left\vert \lambda_{1}\right\vert >1>\left\vert \lambda
_{2}\right\vert .$ Then for any $x\in I$ the Lyapunov orbit starting from
$q(x),$ restricted to the constant energy manifold $M(H(q(x)))$, is a
hyperbolic orbit.
\end{lemma}

\begin{proof}
Let us fix some $\hat{x}\in I.$ For our assumptions to hold, $A$ from
(\ref{eq:A-matrix-def}) needs to be properly defined. This means that
$\frac{\partial H}{\partial x}(q(\hat{x}))\neq0$. By the implicit function
theorem there exists a function $x(p_{x},p_{y})$ with $x(0,\kappa(\hat
{x}))=\hat{x}$ such that $H(x(p_{x},p_{y}),0,p_{x},p_{y})=H(q(\hat{x}))$ and%
\begin{equation}
\left(
\begin{array}
[c]{cc}%
\frac{\partial x}{\partial p_{x}} & \frac{\partial x}{\partial p_{y}}%
\end{array}
\right)  \left(  0,\kappa(\hat{x})\right)  =-\left(  \frac{1}{\frac{\partial
H}{\partial x}}\left(
\begin{array}
[c]{cc}%
\frac{\partial H}{\partial p_{x}} & \frac{\partial H}{\partial p_{y}}%
\end{array}
\right)  \right)  \left(  0,\kappa(\hat{x})\right)  . \label{eq:implicit-x}%
\end{equation}

The Lyapunov orbit starting from $q(\hat{x})$ is contained in the constant
energy manifold $M(H(q(\hat{x})))$. Let us consider $V=M(H(q(\hat{x}%
)))\cap\{y=0\}$ and a Poincar\'{e} map $\tilde{P}:V\rightarrow V$. In a
neighborhood of $q(\hat{x})$ the manifold $V$ can be parameterized by $\left(
p_{x},p_{y}\right)  .$ Since%
\[
\tilde{P}(p_{x},p_{y})=\pi_{(p_{x},p_{y})}P(x(p_{x},p_{y}),p_{x},p_{y})
\]
we have%
\begin{align}
&  D\tilde{P}\left(  0,\kappa(\hat{x})\right) \label{eq:DP-energy}\\
&  =\left(  \left(  \pi_{(p_{x},p_{y})}\frac{\partial P}{\partial x}\right)
\left(
\begin{array}
[c]{cc}%
\frac{\partial x}{\partial p_{x}} & \frac{\partial x}{\partial p_{y}}%
\end{array}
\right)  +\left(
\begin{array}
[c]{cc}%
dP_{2\,2} & dP_{2\,3}\\
dP_{3\,2} & dP_{3\,3}%
\end{array}
\right)  \right)  \left(  \hat{x},0,\kappa(\hat{x})\right)  .\nonumber
\end{align}
By (\ref{eq:implicit-x}), (\ref{eq:DP-energy}) and our assumption about the
spectrum of $B$ of from (\ref{eq:matrix-B}), follows that $\left(
0,\kappa(\hat{x})\right)  $ is a hyperbolic fixed point for the map $\tilde
{P}.$ This means that the Lyapunov orbit starting from $q(\hat{x})$,
restricted to the constant energy manifold $M(H(q(\hat{x})))$ is hyperbolic.
\end{proof}

\begin{remark}
Since $B$ from (\ref{eq:matrix-B}) is a $2\times2$ matrix, estimation of its
eigenvalues is straightforward. Here we profit from the the reduction of
dimension made by restricting to a constant energy manifold.
\end{remark}

Since we consider a small part of the family of orbits (\ref{eq:x0-I}), we can
obtain a much tighter enclosure of the curve $\kappa(x)$ for $x\in I$ than
from Theorem \ref{th:main-1}. Let
\begin{equation}%
\begin{array}
[c]{lll}%
p_{y}^{0}= & -0.836804179646973\quad & J_{0}=p_{y}^{0}+[-1,1]\cdot10^{-13}\\
a= & -4.506866203376769 & J_{1}=p_{y}^{0}+[-1,1]\cdot10^{-12}%
\end{array}
\label{eq:p0y}%
\end{equation}
and%
\begin{equation}
U=\left\{  (x,0,0,p_{y})|x\in I,p_{y}=a\left(  x-x^{0}\right)  +\iota,\iota\in
J_{1}\right\}  . \label{eq:U-prop-for-I}%
\end{equation}

\begin{proposition}
\label{prop:kappa-bound} For $x\in I$, with $I$ from (\ref{eq:x0-I}), we have
$q(x)=(x,0,0,\kappa(x))\subset U$ and
\begin{align}
\kappa^{\prime}(x)  &  \in[-4.506980818,-4.506751634], \label{eq:dkappa-bound}%
\\
\frac{d}{dx}H(q(x))  &  \in[-0.3670937615,-0.3670674516], \label{eq:dH-bound}%
\end{align}
\begin{align}
H(\underline{I},0,0,a(\underline{I}-x_{0})+J_{1} )  &  \in
[-1.514999999635,-1.514999999631],\nonumber\\
H(\overline{I},0,0,a(\overline{I}-x_{0})+J_{1})  &  \in
[-1.515000000369,-1.515000000365].\nonumber
\end{align}
Moreover, the orbits (when considered on their constant energy manifolds) are
hyperbolic, and we have following bounds for the eigenvalues
\begin{align}
\lambda_{1}  &  \in\left[  1450.24,1481.68\right]  ,\label{eq:lambda-bounds}\\
\lambda_{2}  &  \in10^{-4}\left[  6.74909,6.89541\right]  .\nonumber
\end{align}

\end{proposition}

\begin{proof}
The proof was performed with computer assistance. It required no subdivision
of $U$ and the computation took less than two seconds on a standard laptop.

Existence of $q(x)\subset U$ was shown using Lemma \ref{lem:Newton-per-orb}.
From it also follows the bound (\ref{eq:dkappa-bound}) for $\kappa^{\prime
}(x).$ The bound (\ref{eq:dH-bound}) follows from Remark
\ref{rem:energy-foliation}. Hyperbolicity and bounds (\ref{eq:lambda-bounds})
follow from Lemma \ref{lem:Normal-Hyperbolicity}.
\end{proof}

\section{Cone Conditions and Bounds for Unstable Manifolds of Saddle-Center
Fixed Points\label{sec:unstable-topological}}

In this section we provide a topological tool that can be used for
rigorous-computer-assisted detection of unstable manifolds of saddle-center
fixed points. The method is a modification of \cite{Z1}, where instead of
saddle-center a standard hyperbolic fixed point was considered. The content of
this section is a general result. In section \ref{sec:parameterization-fibers}
we return to the PRC3BP and show how to apply it for the proof of Theorem
\ref{th:main}.

Let $F:\mathbb{R}^{n}\rightarrow\mathbb{R}^{n}$ be a $C^{k}$ diffeomorphism
with a fixed point $v^{\ast}\in\mathbb{R}^{n}$ and $k\geq1.$ Assume that for
eigenvalues $\lambda_{1},\lambda_{2},\ldots,\lambda_{n}$ from the spectrum of
$DF(v^{\ast})$ we have%
\begin{align}
\left\vert \text{re}\lambda_{1}\right\vert  &  >m>1,\label{eq:lambda-m}\\
\left\vert \text{re}\lambda_{i}\right\vert  &  <m\quad\text{for }%
i=2,\ldots,n.\nonumber
\end{align}
Let $W^{u}(v^{\ast})$ denote the unstable manifold of $v^{\ast}$ associated
with the eigenvalue $\lambda_{1}$
\[
W^{u}\left(  v^{\ast}\right)  =\left\{  v|\left\Vert F^{-n}(v)-v^{\ast
}\right\Vert <Cm^{-n}\text{ for all }n\in\mathbb{N}\text{ and some
}C>0\right\}  .
\]
Let $u=1$ and $c=n-1$. The notations $u$ and $c$ will stand for "unstable" and
"central" coordinates of $F$ at $v^{\ast}$. Consider two balls $B_{u}$ and
$B_{c},$ of dimensions $u$ and $c$ respectively, such that $B_{u}\times B_{c}$
is centered at $v^{\ast}$. For a point $v\mathbf{\in}\mathbb{R}^{u}%
\times\mathbb{R}^{c}$ we shall write $v=\left(  \mathsf{x},\mathsf{y}\right)
,$ with $\mathsf{x}\in\mathbb{R}^{u},$ $\mathsf{y}\mathbf{\in}\mathbb{R}^{c}.$
In these notations we shall also write the fixed point as $v^{\ast}=\left(
\mathsf{x}^{\ast},\mathsf{y}^{\ast}\right)  .$

\begin{remark}
We do not need to assume that $(\mathsf{x},0)$ is the eigenvector associated
with $\lambda_{1}$ and that vectors $(0,\mathsf{y})$ span the eigenspace of
$\lambda_{2},\ldots,\lambda_{n}$. For our method to work it is enough if these
vectors are "roughly" aligned with the eigenspaces. This is important for us,
since in any computer assisted computation it is usually not possible to
compute the eigenvectors with full precision.
\end{remark}

Let $\alpha\in\mathbb{R,}$ $\alpha>0$ and consider a function $Q:\mathbb{R}%
^{u}\times\mathbb{R}^{c}\rightarrow\mathbb{R}$%
\[
Q(\mathsf{x},\mathsf{y})=\alpha\mathsf{x}^{2}-\left\Vert \mathsf{y}\right\Vert
^{2}.
\]
For $v_{0}\in\mathbb{R}^{u}\times\mathbb{R}^{c}$ we shall use a notation
$Q^{+}(v_{0})$ for a cone
\[
Q^{+}(v_{0})=\left\{  v|Q(v-v_{0})\geq0\right\}  .
\]
Let us assume that $\alpha$ is chosen sufficiently small so that
$Q^{+}(v^{\ast})\cap B_{u}\times B_{c}$ does not intersect with $B_{u}%
\times\partial B_{c}$ (See Figure \ref{fig:Wu-proof}).

\begin{definition}
We shall say that $h:B_{u}\rightarrow B_{u}\times B_{c}$ is a \emph{horizontal
disc in }$B_{u}\times B_{c}$\emph{ for cones given by }$Q$ if $h(\mathsf{x}%
^{\ast})=v^{\ast},$ $\pi_{\mathsf{x}}h(\mathsf{x})=\mathsf{x}$ and for any
$\mathsf{x}_{1}\neq\mathsf{x}_{2}$ holds $Q\left(  h(\mathsf{x}_{1}%
)-h(\mathsf{x}_{2})\right)  >0.$
\end{definition}

\begin{figure}[h]
\begin{center}
\includegraphics[height=1.1in]{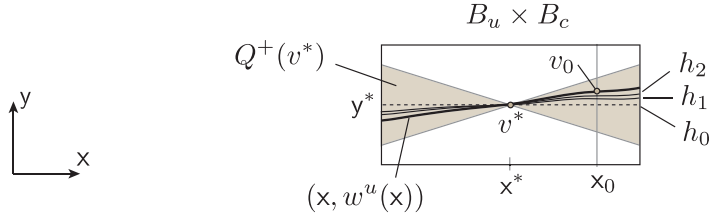}
\end{center}
\caption{Construction of the curve $(\mathsf{x},w^{u}(\mathsf{x}))$ which lies
on the unstable manifold of $v^{\ast}$.}%
\label{fig:Wu-proof}%
\end{figure}

\begin{lemma}
\label{lem:cc-wu-bound}Assume that for any $v_{1},v_{2}\in Q^{+}(v^{\ast}),$
such that $Q(v_{1}-v_{2})\geq0,$ we have%
\begin{equation}
Q(F(v_{1})-F(v_{2}))>0.\label{eq:cc1}%
\end{equation}
Let $m$ be the constant from (\ref{eq:lambda-m}). If for any $v\in B,$ $v\neq
v^{\ast},$ $Q(v-v^{\ast})\geq0$ holds%
\begin{equation}
\left\Vert F(v)-v^{\ast}\right\Vert >m\left\Vert v-v^{\ast}\right\Vert
,\label{eq:cc2}%
\end{equation}
then $W^{u}\left(  v^{\ast}\right)  \subset Q^{+}(v^{\ast}).$ Moreover, there
exists a function $w^{u}:B_{u}\rightarrow B_{c}$ such that $(id,w^{u}%
)(B_{u})=W^{u}(v^{\ast})\cap U,$ and for any $\mathsf{x}_{1},\mathsf{x}_{2}\in
B_{u},$ $\mathsf{x}_{1}\neq\mathsf{x}_{2}$
\begin{equation}
Q((\mathsf{x}_{1},w^{u}\left(  \mathsf{x}_{1}\right)  )-(\mathsf{x}_{2}%
,w^{u}\left(  \mathsf{x}_{2}\right)  ))>0\label{eq:wu-cc}%
\end{equation}
and%
\begin{equation}
\left\Vert \left(  w^{u}\right)  ^{\prime}\left(  \mathsf{x}\right)
\right\Vert \leq\sqrt{\alpha}\quad\text{for all }\mathsf{x}\in B_{u}%
.\label{eq:wu-der}%
\end{equation}

\end{lemma}

\begin{proof}
We shall first show that for any $\mathsf{x}_{0}\in B_{u}\backslash
\{\mathsf{x}^{\ast}\}$ there exists a point $v_{0}=(\mathsf{x}_{0}%
,w^{u}(\mathsf{x}_{0}))\in Q^{+}(v^{\ast})$ such that $v_{0}\in W^{u}(v^{\ast
}).$ Let $h_{0}(\mathsf{x})=(\mathsf{x},\mathsf{y}^{\ast})$ be a horizontal
disc (See Figure \ref{fig:Wu-proof}). Observe that $F(h_{0}(\mathsf{x}^{\ast
}))=F(\mathsf{x}^{\ast},\mathsf{y}^{\ast})=v^{\ast}$. By assumptions
(\ref{eq:cc1}), (\ref{eq:cc2}) the curve $F(h_{0}(\mathsf{x}))$ is contained
in $Q^{+}(v^{\ast})$ and $F(h_{0}(\partial B_{u}))\cap B_{u}\times
B_{c}=\emptyset$. Moreover, by assumption (\ref{eq:cc1}) for any
$\mathsf{x}_{1},\mathsf{x}_{2}\in B_{u},$ $\mathsf{x}_{1}\neq\mathsf{x}_{2}$
\[
Q(F(h_{0}(\mathsf{x}_{1}))-F(h_{0}(\mathsf{x}_{2})))>Q(h_{0}(\mathsf{x}%
_{1})-h_{0}(\mathsf{x}_{2}))>0,
\]
which means that $\{F(h_{0}(\mathsf{x}))|\mathsf{x}\in B_{u}\}\cap B_{u}\times
B_{c}$ is a graph of a horizontal disc. Let us denote this disc by $h_{1}$ and
observe that $h_{1}(\mathsf{x}^{\ast})=v^{\ast}.$ In other words, let $h_{1}$
be the graph transform of the disc $h_{0}.$

Taking $F(h_{1}(\mathsf{x}))$ and applying an identical argument, we observe
that
\[
\{F(h_{1}(\mathsf{x}))|\mathsf{x}\in B_{u}\}\cap B_{u}\times B_{c}%
\]
is a graph of a horizontal disc $h_{2}.$ Repeating this procedure we can
construct a sequence of horizontal discs $h_{0},h_{1},h_{2},\ldots$. For a
fixed $\mathsf{x}_{0}$, due to compactness of closure of $B_{c},$ there exists
a subsequence $h_{k_{i}}(\mathsf{x}_{0})$ convergent to some point $v_{0}\in
B_{u}\times\mathrm{cl}B_{c}$. For any $i,n\in\mathbb{N}$ with $k_{i}>n$ the
point $F^{-n}(h_{k_{i}}(\mathsf{x}_{0}))$ lies on the graph of $h_{k_{i}-n}$
and hence is also in $Q^{+}(v^{\ast}).$ This means that for any $n\in
\mathbb{N}$
\[
F^{-n}(v_{0})=\lim_{i\rightarrow\infty}F^{-n}(h_{k_{i}}(\mathsf{x}_{0}))\in
Q^{+}(v^{\ast}).
\]
By assumption (\ref{eq:cc2}) we have%
\[
\left\Vert F^{-n}(v_{0})-v^{\ast}\right\Vert <\frac{1}{m^{n}}\left\Vert
v_{0}-v^{\ast}\right\Vert ,
\]
which means that $v_{0}\in W^{u}\left(  v^{\ast}\right)  $. By construction
$\pi_{\mathsf{x}}v_{0}=\mathsf{x}_{0}$, hence we can define $w^{u}%
(\mathsf{x}_{0}):=\pi_{\mathsf{y}}v_{0}$. 

By the stable/unstable manifold theorem, there exists a small interval
$I_{\varepsilon}=(\mathsf{x}^{\ast}-\varepsilon,\mathsf{x}^{\ast}%
+\varepsilon)$ in which $\{(\mathsf{x},w^{u}(\mathsf{x}))|\mathsf{x}\in
I_{\varepsilon}\}$ is a $C^{k}$ curve which gives full description of
$W^{u}\left(  v^{\ast}\right)  .$ Since $(\mathsf{x},w^{u}(\mathsf{x}))\subset
Q^{+}(v^{\ast})$ we have $(1,(w^{u})^{\prime}(\mathsf{x}^{\ast}))\in
Q^{+}(0).$ Since for sufficiently small $\varepsilon$ the vector $(1,\left(
w^{u}\right)  ^{\prime}(\mathsf{x}))$ is arbitrarily close to $(1,(w^{u}%
)^{\prime}(\mathsf{x}^{\ast})),$ for $\mathsf{x}_{1},\mathsf{x}_{2}\in
I_{\varepsilon}$%
\begin{equation}
Q\left(  (\mathsf{x}_{1},w^{u}(\mathsf{x}_{1}))-(\mathsf{x}_{2},w^{u}%
(\mathsf{x}_{2}))\right)  >0. \label{eq:Q-temp}%
\end{equation}
Iterating the curve $(\mathsf{x},w^{u}(\mathsf{x}))$ through $F$, by
(\ref{eq:cc1}), (\ref{eq:cc2}) we obtain our function $w^{u}:B_{u}\rightarrow
B_{c}$. Note that by our construction for any $\mathsf{x}_{1},\mathsf{x}%
_{2}\in B_{u}$ inequality (\ref{eq:Q-temp}) holds. This implies that for any
$\mathsf{x}_{1},\mathsf{x}_{2}\in B_{u}$%
\[
\frac{\left\Vert w^{u}(\mathsf{x}_{1})-w^{u}(\mathsf{x}_{2})\right\Vert ^{2}%
}{\left\vert \mathsf{x}_{1}-\mathsf{x}_{2}\right\vert ^{2}}<\alpha,
\]
which in turn gives (\ref{eq:wu-der}).
\end{proof}

\begin{remark}
Lemma \ref{lem:cc-wu-bound} can easily be generalized to higher dimension of
$W^{u}(v^{\ast})$. The proof would be identical, taking $Q(\mathsf{x}%
,\mathsf{y})=\alpha\Vert\mathsf{x}\Vert^{2}-\Vert\mathsf{y}\Vert^{2}$. Here we
have set up our discussion so that $W^{u}(v^{\ast})$ is one dimensional simply
because this is what we shall need for our application to the PRC3BP.
\end{remark}

\begin{remark}
By taking the inverse map, Lemma \ref{lem:cc-wu-bound} can be used to prove
existence of stable manifolds.
\end{remark}

To verify assumptions (\ref{eq:cc1}) and (\ref{eq:cc2}) in practice, it is
best to make use of an interval matrix $\mathbf{A}=[DF(Q^{+}(v^{\ast}))].$
Then for any $v_{1},v_{2}\in Q^{+}(v^{\ast})$ we have%
\begin{equation}
F(v_{1})-F(v_{2})=\int_{0}^{1}DF\left(  v_{2}+t\left(  v_{1}-v_{2}\right)
\right)  dt\cdot\left(  v_{1}-v_{2}\right)  \in\mathbf{A}\left(  v_{1}%
-v_{2}\right)  .\label{eq:Fx1x2}%
\end{equation}
This means that
\begin{equation}
F(v)-v^{\ast}\subset\mathbf{A}\left(  v-v^{\ast}\right)
.\label{eq:cc2-interval-matrix}%
\end{equation}
To verify (\ref{eq:cc2}) using (\ref{eq:cc2-interval-matrix}) we can apply
Lemma \ref{lem:cc2-expansion-condition} from the Appendix.

Let us now turn to verification of (\ref{eq:cc1}). Let $C_{Q}$ be a diagonal
matrix such that $v^{T}C_{Q}v=Q(v).$ Equation (\ref{eq:Fx1x2}) gives an
estimate%
\begin{equation}
Q\left(  F(v_{1})-F(v_{2})\right)  \subset\left(  v_{1}-v_{2}\right)
^{T}\mathbf{A}^{T}C_{Q}\mathbf{A}\left(  v_{1}-v_{2}\right)  .
\label{eq:cc1-interval-matrix}%
\end{equation}
To verify (\ref{eq:cc1}) using (\ref{eq:cc1-interval-matrix}) we can apply
Lemma \ref{lem:cc1-matrix} from the Appendix.
%\usepackage{pdfsync}
%\usepackage{hyperref}%
%\usepackage{pdfsync}
%\usepackage{hyperref}%
%\usepackage{pdfsync}
%\usepackage{hyperref}%
%\usepackage{pdfsync}
%\usepackage{hyperref}%
%\usepackage{pdfsync}
%\usepackage{hyperref}%

%TCIDATA{OutputFilter=latex2.dll}
%TCIDATA{Version=5.00.0.2606}
%TCIDATA{LaTeXparent=0,0,online-edit.tex}

\section{Rigorous Bounds for Invariant Manifolds associated with Lyapunov
Orbits\label{sec:parameterization-fibers}}

In this section we give a proof of Theorem \ref{th:main}. In sections
\ref{sec:param} and \ref{sec:fiber-enclosure} we shall show how to apply the
method from section \ref{sec:unstable-topological} to detect fibers of
unstable manifolds of Lyapunov orbits. In section \ref{sec:transversality} we
shall show how to prove that the manifolds intersect. Using these results, in
section \ref{sec:proof-of-main-th} we give a computer assisted proof Theorem
\ref{th:main}.

\subsection{Parameterization Method\label{sec:param}}

The method from section \ref{sec:unstable-topological} requires a good change
of coordinates which "straightens out" the unstable manifold. We shall obtain
such a change of coordinates using a parameterization method. In this
subsection we give an outline of this procedure.

In this section we shall fix some $x\in I$ and show how to find an unstable
fiber of a point
\[
q_{0}=q(x)=\left(  x,0,0,\kappa(x)\right)  \in L(x).
\]
We shall use a notation $\tau=\tau(q_{0})$ for the return time along the
trajectory. The point $q_{0}$ is a saddle center fixed point for a $\tau$-time
map $\Phi_{\tau}:\mathbb{R}^{4}\rightarrow\mathbb{R}^{4}.$

Let $C$ denote a matrix which brings $D\Phi_{\tau}\left(  q_{0}\right)  $ to
real Jordan form. By $\tilde{\Phi}_{\tau}:\mathbb{R}^{4}\rightarrow
\mathbb{R}^{4}$ we shall denote the time $\tau$ map in the linearized local
coordinates%
\[
\tilde{\Phi}_{\tau}\left(  v\right)  :=C^{-1}\left(  \Phi_{\tau}\left(
q_{0}+Cv\right)  -q_{0}\right)  .
\]

Let $W^{u}(\tilde{\Phi}_{\tau},0)$ denote the unstable manifold of
$\tilde{\Phi}_{\tau}$ at zero. If we can find a function%
\[
K=\left(  K_{0},K_{1},K_{2},K_{3}\right)  :\mathbb{R\rightarrow R}^{4},
\]
which for all $\mathsf{x}$ in an interval $I_{0}=[\underline{\mathsf{x}%
},\overline{\mathsf{x}}],$ $\underline{\mathsf{x}}<0<\overline{\mathsf{x}},$
is a solution of a cohomology equation
\begin{equation}
\tilde{\Phi}_{\tau}\left(  K(\mathsf{x})\right)  =K(\lambda\mathsf{x}),
\label{eq:cohomology-equation}%
\end{equation}
then $K(\mathsf{x})\subset W^{u}(\tilde{\Phi}_{\tau},0)$ for $\mathsf{x}\in
I_{0}.$

Once $K$ is established we can consider a nonlinear change of coordinates%
\[
\psi=\left(  \psi_{0},\psi_{1},\psi_{2},\psi_{3}\right)  :\mathbb{R}%
^{4}\rightarrow\mathbb{R}^{4}%
\]
defined as%
\begin{align}
\psi_{0}\left(  \mathsf{x},\mathsf{y}_{1},\mathsf{y}_{2},\mathsf{y}%
_{3}\right)   &  =K_{0}(\mathsf{x})-\left(  \mathsf{y}_{1}K_{1}^{\prime
}(\mathsf{x})+\mathsf{y}_{2}K_{2}^{\prime}(\mathsf{x})+\mathsf{y}_{3}%
K_{3}^{\prime}(\mathsf{x})\right)  ,\label{eq:psi-form}\\
\psi_{i}\left(  \mathsf{x},\mathsf{y}_{1},\mathsf{y}_{2},\mathsf{y}%
_{3}\right)   &  =K_{i}(\mathsf{x})+\mathsf{y}_{i}K_{0}^{\prime}%
(\mathsf{x})\quad\text{for }i=1,2,3.\nonumber
\end{align}
Note that $\psi(\mathsf{x},0)=K(\mathsf{x})$ gives points on the unstable
manifold of the fixed point for the map $\tilde{\Phi}_{\tau}.$ The intuitive
idea behind (\ref{eq:psi-form}) is to orthogonalize coordinates around
$K(\mathsf{x})$ (see Figure \ref{fig:psi}).

Let us define a local map%
\[
F=\psi^{-1}\circ\tilde{\Phi}_{\tau}\circ\psi.
\]
Such map will play the role of $F$ from section \ref{sec:unstable-topological}%
. Observe that%
\[
\{C\psi\left(  K(\mathsf{x})\right)  +q_{0}|\mathsf{x}\in I_{0}\}\subset
C\psi\left(  W^{u}\left(  F,0\right)  \right)  +q_{0}=W^{u}\left(  \Phi_{\tau
},q_{0}\right)  \subset W^{u}(L(x_{0})).
\]

\begin{figure}[h]
\begin{center}
\includegraphics[height=2.6cm]{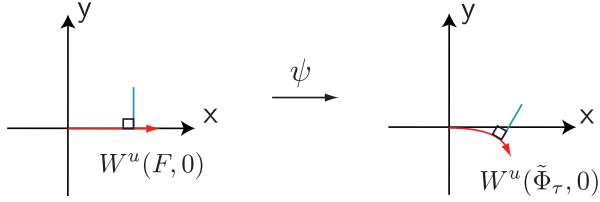}
\end{center}
\caption{The nonlinear change of coordinates $\psi.$}%
\label{fig:psi}%
\end{figure}

\subsection{Bounds for Unstable Fibers through Parameterization and Cone
conditions\label{sec:fiber-enclosure}}

The map $\psi$ (\ref{eq:psi-form}) gives us a change of coordinates which
locally "straightens out" the unstable manifold. The problem with applying the
procedure from section \ref{sec:param} in practice lies in the fact that
usually finding an analytic formula for $K$ satisfying
(\ref{eq:cohomology-equation}) is impossible. The best that can be done is to
find a $K$ which is a polynomial approximation of a solution of
(\ref{eq:cohomology-equation}). This can be done by expanding $\tilde{\Phi
}_{\tau}$ into a Taylor series and inductively comparing the coefficients in
(\ref{eq:cohomology-equation}) (for a detailed description of this method we
refer the reader to \cite{CFL}; see in particular Section 4 and Theorem 4.1).
If we find such an approximate solution of (\ref{eq:cohomology-equation}),
then the set $\{(\mathsf{x},0)|\mathsf{x}\in I_{0}\}$ is no longer the
unstable manifold for $F$ (defined by (\ref{eq:F-loc})), but its
approximation. Even though then our description of the unstable fiber is not
entirely accurate, we can apply the method from Section
\ref{sec:unstable-topological} to obtain a rigorous enclosure of $W^{u}\left(
F,0\right)  $. This enclosure can then be transported to the original coordinates.

In this Section we shall assume that $q_{0}$ is an arbitrary point close to
$q(x)=(x,0,0,\kappa(x))$ for $x\in I,$ $C$ is some given matrix and
$K:\mathbb{R\rightarrow R}^{4}$ is some given polynomial and that $\psi$ is
defined by (\ref{eq:psi-form}).

\begin{remark}
\label{rem:Coord-choice1} Let us stress that the point $q_{0}$ is a numerical
approximation of $q(x),$ the matrix $C$ will be a (non-rigorous) numerically
obtained estimate for the change into Jordan form of the map $\Phi_{\tau}$. We
do not assume that this change is rigorously computed. This is practically
impossible due to the fact that we do not have an analytic formula for
$D\Phi_{\tau}\left(  q_{0}\right)  $. For us the matrix $C$ is simply some
approximation of the matrix which takes $D\Phi_{\tau}\left(  q_{0}\right)  $
into Jordan form. Let us note that it is not difficult to find an interval
matrix $\mathbf{C}^{-1}$ such that the inverse matrix of our $C$ is contained
in $\mathbf{C}^{-1}$.
\end{remark}

\begin{remark}
\label{rem:Coord-choice2} In our setting the polynomial $K$ is an
approximation of the solution of (\ref{eq:cohomology-equation}). In practice
we cannot obtain a fully rigorous solution of (\ref{eq:cohomology-equation}).
It is important to emphasize that we also do not have an inverse of $\psi$. It
is also not simple to find good rigorous estimates for the function $\psi
^{-1}$ due to the fact that $K$ is a high order polynomial. We shall therefore
set up all our subsequent computations so that we will never need to use the
inverse function of $\psi$.
\end{remark}

For $x\in I,$ let $\tau(x)$ be the period of an orbit $L(x)\subset\Lambda$. We
define a map%
\begin{equation}
F=\psi^{-1}\circ\tilde{\Phi}_{\tau(x)}\circ\psi. \label{eq:F-loc}%
\end{equation}
Note that for each $x\in I$ we have a different map $F$. We omit this in our
notation for simplicity, and also because below methods for obtaining rigorous
bounds for $F$ and its derivative shall work for all $x\in I$.

We shall first be interested in computing rigorous bounds for $F(U)$. It turns
out that (\ref{eq:F-loc}) is impossible to apply since we do not have a
formula for $\psi^{-1}$. Even if we did, direct application of (\ref{eq:F-loc}%
) in interval arithmetic would provide very bad estimates due to strong
hyperbolicity of the map. We use a more subtle method.

We shall first need the following notations. Let $\boldsymbol{T}$ denote an
interval such that $\tau(x)\in\boldsymbol{T}$ for all $x\in I$. Let
$\lambda\in\mathbb{R}$ be some number close to an unstable eigenvalue of
$D\Phi_{\tau(x)}(q(x))$ for some $x\in I$. We shall slightly abuse notations
and also consider $\lambda:\mathbb{R}^{4}\rightarrow\mathbb{R}^{4}$ as a
function defined on $v=(\mathsf{x},\mathsf{y})\in\mathbb{R}\times
\mathbb{R}^{3}$ as%
\[
\lambda(\mathsf{x},\mathsf{y}):=\left(  \lambda\mathsf{x},\mathsf{y}\right)
.
\]

The following Lemma allows us to obtain rigorous bounds on pre-images of $F$
from (\ref{eq:F-loc}).

\begin{lemma}
\label{lem:F-image}Let $U_{1}\subset\mathbb{R}^{4}$ be a given set. Let
$G:\mathbb{R}\times\mathbb{R}^{4}\times\mathbb{R}^{4}\rightarrow\mathbb{R}%
^{4}$ be defined as%
\begin{equation}
G(\tau,v_{1},v_{2})=\Phi_{\tau}\left(  C\psi\left(  v_{1}\right)
+q_{0}\right)  -\left(  C\psi(\lambda(v_{2}))+q_{0}\right)  .
\label{eq:G-2dim}%
\end{equation}
Let $U_{2}\subset\mathbb{R}^{4}$ be a set and $\mathbf{A}\left(  U_{2}\right)
$ be an interval matrix defined as%
\[
\mathbf{A}\left(  U_{2}\right)  =-\left[  CD\psi\left(  \lambda\left(
U_{2}\right)  \right)  D\lambda\right]  .
\]
If%
\begin{equation}
N(\mathbf{T},v_{0},U_{1},U_{2}):=v_{0}-\left[  \mathbf{A}\left(  U_{2}\right)
\right]  ^{-1}\left[  G(\mathbf{T},U_{1},v_{0})\right]  \subset U_{2},
\label{eq:Floc-Newton-assumption}%
\end{equation}
then
\begin{equation}
F(U_{1})\subset\lambda\left(  U_{2}\right)  . \label{eq:F-image}%
\end{equation}

\end{lemma}

\begin{proof}
The proof is given in the Appendix in section \ref{sec:Local-map-F(U)}. See
also Remark \ref{rem:preimage} for comments on practical application of the lemma.
\end{proof}

\begin{remark}
The choice of the function $G$ is motivated by the following diagram.%
\[%
\begin{array}
[c]{ccc}%
\mathbb{R}^{4} & \overset{\Phi_{\tau}}{\longrightarrow} & \mathbb{R}^{4}\\
\quad\quad\uparrow C+q_{0} &  & \quad\quad\uparrow C+q_{0}\\
\mathbb{R}^{4} & \overset{\tilde{\Phi}_{\tau}}{\longrightarrow} &
\mathbb{R}^{4}\\
\uparrow\psi &  & \uparrow\psi\\
\mathbb{R}^{4} & \left(  \overset{\lambda}{\longrightarrow}\right)  &
\mathbb{R}^{4}%
\end{array}
\]
The diagram is not fully commutative, hence the bracket for $\lambda$.
Intuitively, for $v=(\mathsf{x},0)\in\mathbb{R}\times\mathbb{R}^{3}$ the
diagram should "almost commute". Even though this statement is nowhere close
to rigorous, it might make the method and proof of Lemma \ref{lem:F-image}
more intuitive.
\end{remark}

We now turn to the computation of rigorous bounds for the derivatives of
(\ref{eq:F-loc}). For any $(\mathsf{x},\mathsf{y})$ contained in a set
$B\subset\mathbb{R}^{4}$ we have the following estimates
\begin{align}
DF(\mathsf{x},\mathsf{y})  &  =\left(  D\psi\left(  F(\mathsf{x}%
,\mathsf{y})\right)  \right)  ^{-1}C^{-1}D\Phi_{\tau(x)}\left(  C\psi
(\mathsf{x},\mathsf{y})+q_{0}\right)  CD\psi(\mathsf{x},\mathsf{y}%
)\label{eq:DFloc-form}\\
&  \subset\left[  \left(  D\psi\left(  F(B)\right)  \right)  ^{-1}\right]
\cdot\left[  C^{-1}\right]  \cdot\left[  D\Phi_{\mathbf{T}}\left(
C\psi(B)+q_{0}\right)  \right]  \cdot C\cdot\left[  D\psi\left(  B\right)
\right] \nonumber\\
&  =:\left[  DF(B)\right]  .\nonumber
\end{align}

Note that to compute $\left[  DF(B)\right]  $ from (\ref{eq:DFloc-form}) we do
not need to use $\psi^{-1}.$

\begin{remark}
Using Lemma \ref{lem:F-image} and (\ref{eq:DFloc-form}) we can in practice compute rigorous
bounds for $[F(B)]$ and $[DF(B)].$ We perform such computations in Section
\ref{sec:Rig-fibers} with the use of CAPD library (http://capd.ii.uj.edu.pl/).
The library allows for computation of rigorous estimates for $\Phi
_{\mathbf{T}}$ and its derivative and for rigorous-enclosure operations on
maps and interval matrixes.
\end{remark}

Proposition \ref{eq:U-prop-for-I} gives a bound on a set $U$
(\ref{eq:U-prop-for-I}) which contains all fixed points $q(x),$ with $x\in I$,
of the map $\Phi_{\tau(x)}$. This set can be transported to local coordinates
$(\mathsf{x},\mathsf{y})$. Let $B_{0}\subset\mathbb{R}^{4}$ be such set that%
\[
\{\psi^{-1}(C^{-1}(q(x)-q_{0}))\}|x\in I\}\subset B_{0}.
\]
Such set can easily be computed using for example a technical Lemma
\ref{lem:Newton-preimage} from the Appendix.

Taking a four dimensional set (see Figure \ref{fig:Loc-bound})
\[
B=\bigcup_{v\in B_{0}}Q^{+}(v)\subset
\mathbb{R}^{4}
\] 
using (\ref{eq:DFloc-form})
and Lemmas \ref{lem:cc2-expansion-condition}, \ref{lem:cc1-matrix} to verify
assumptions of Lemma \ref{lem:cc-wu-bound}, we can obtain a bound for the
unstable fibers of all $q(x)$ for $x\in I.$ The obtained bound is computed in
local coordinates $(\mathsf{x},\mathsf{y}),$ but can easily be transported
back to the original coordinates $(x,y,p_{x},p_{y})$ of the system. Detailed
results of such computation will be presented in section \ref{sec:Rig-fibers}.

\begin{figure}[h]
\begin{center}
\includegraphics[height=0.8in]{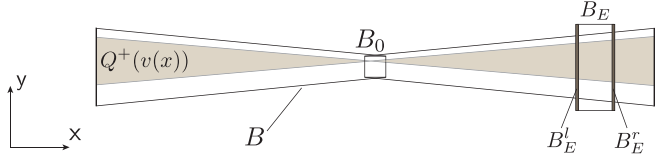}
\end{center}
\caption{Local bound on the unstable manifold. Each fixed point $\psi^{-1}(C^{-1}(q(x)-q_0)),$ for
$x\in I$, lies in $B_{0}$ and its unstable manifold is contained in
$B=\bigcup_{v\in B_{0}}Q^{+}(v).$}%
\label{fig:Loc-bound}%
\end{figure}

\begin{remark}
Let us emphasize that to apply the method it is enough to use a single point
$q_{0},$ single matrix $C$ and single nonlinear change $\psi$. It is not
necessary to use different changes to local coordinates for different $x\in I$.
\end{remark}

\begin{remark}
Let us note that from the fact that $W^{s}(L(x))$ is $S$-symmetric to
$W^{u}(L(x)),$ without any effort we also obtain mirror bounds for fibers of
$W^{s}(L(x))$ .
\end{remark}

\subsection{Transversal Intersections of Manifolds\label{sec:transversality}}

In this section we discuss how the bounds for fibers of $q(x)$ discussed in
section \ref{sec:fiber-enclosure} can be used to prove transversal
intersections of manifolds $W^{u}(L(x))$ and $W^{s}(L(x))$ for $L(x)\subset
\Lambda$ (see (\ref{eq:Lambda-def}) for definition of $\Lambda$).

Let $\mathsf{x}^{l},\mathsf{x}^{r}\in\mathbb{R}$ be such that $\mathsf{x}%
^{l}<\mathsf{x}^{r}$ and $\pi_{\mathsf{x}}B_{0}<\mathsf{x}^{l},\mathsf{x}%
^{r}.$ Let $B_{c}\subset\mathbb{R}^{3}$ be such that $\pi_{\mathsf{y}}B\subset
B_{c}$. Let $B_{E},B_{E}^{l},B_{E}^{r}$ be defined as (see Figure
\ref{fig:Loc-bound})%
\begin{align*}
B_{E}  &  =[\mathsf{x}^{l},\mathsf{x}^{r}]\times B_{c},\\
B_{E}^{l}  &  =\left\{  \mathsf{x}^{l}\right\}  \times B_{c},\\
B_{E}^{r}  &  =\left\{  \mathsf{x}^{r}\right\}  \times B_{c},
\end{align*}
and let%
\begin{align*}
V^{+}  &  =\{(\mathsf{x},\mathsf{y}_{1},\mathsf{y}_{2},\mathsf{y}_{3}%
)\in\mathbb{R}^{4}|\mathsf{x}=1,\mathsf{y}_{i}\in\lbrack-\sqrt{\alpha}%
,\sqrt{\alpha}]\text{ for }i=1,2,3\},\\
V^{-}  &  =\{(\mathsf{x},\mathsf{y}_{1},\mathsf{y}_{2},\mathsf{y}_{3}%
)\in\mathbb{R}^{4}|\mathsf{x}=-1,\mathsf{y}_{i}\in\lbrack-\sqrt{\alpha}%
,\sqrt{\alpha}]\text{ for }i=1,2,3\},\\
V  &  =\{\gamma v|v\in V^{+},\gamma\geq0\}\cup\{\gamma v|v\in V^{-},\gamma
\geq0\}.
\end{align*}
Note that%
\[
Q^{+}(0)\subset V.
\]

Consider a section
\[
\Sigma=\{y=0\}\cap\{x>0\}\cap\{p_{x}^{2}<2\left(  H(L(x))+\Omega(x,y)\right)
\}.
\]
This shall be a section where we detect the intersection of $W^{u}(L(x))$ and
$W^{s}(L(x))$ (see Figures \ref{fig:plot3d}, \ref{fig:plot2d}). Let $\phi$ be
the flow of (\ref{eq:PRC3BP}) and define%
\begin{gather*}
\tau(q)=\inf\{t>0:\phi(t,q)\in\Sigma\},\\
\mathcal{G}:B_{E}\rightarrow\Sigma,\\
\mathcal{G}\left(  \mathsf{x},\mathsf{y}\right)  =\phi(\tau(C\psi
(\mathsf{x},\mathsf{y})+q^{0}),C\psi(\mathsf{x},\mathsf{y})+q^{0}).
\end{gather*}

\begin{lemma}
\label{lem:transversality}Assume that for $F$ defined in (\ref{eq:F-loc})
assumptions of Lemma \ref{lem:cc-wu-bound} hold. If also
\begin{equation}
\pi_{p_{x}}\mathcal{G}\left(  B_{E}^{l}\right)  <0,\qquad\pi_{p_{x}%
}\mathcal{G}\left(  B_{E}^{r}\right)  >0, \label{eq:exit-set-ass}%
\end{equation}
then for any $x\in I$ (with $I$ defined in (\ref{eq:x0-I})) the manifolds
$W^{u}(L(x))$ and $W^{s}(L(x))$ intersect.

Moreover, if for any $v^{+}\in V^{+}$ and $v^{-}\in V^{-}$%
\begin{align}
\pi_{x}\left[  D\mathcal{G}\left(  B_{E}\right)  \right]  v^{+}  &
>0,\qquad\pi_{p_{x}}\left[  D\mathcal{G}\left(  B_{E}\right)  \right]
v^{+}>0,\label{eq:slope-cond1}\\
\pi_{x}\left[  D\mathcal{G}\left(  B_{E}\right)  \right]  v^{-}  &
<0,\qquad\pi_{p_{x}}\left[  D\mathcal{G}\left(  B_{E}\right)  \right]
v^{-}<0,\nonumber
\end{align}
then for each fixed $x\in I$ the intersection is transversal on the constant
energy manifold $M(H(L(x)))$ (see (\ref{eq:M-energy}) for definition of $M$).
\end{lemma}

\begin{proof}
Let us fix an $x\in I.$ First let us observe that because energy
(\ref{eq:H-PRC3BP}) is preserved, the manifold $M(L(x))\cap\Sigma$ can be
parameterized by $x,p_{x}$ since%
\begin{equation}
p_{y}=p_{y}\left(  x,p_{x}\right)  =\sqrt{2(H(L(x))+\Omega(x,y))-p_{x}^{2}}+x
\label{eq:py-well-def}%
\end{equation}
is well defined.

By Lemma \ref{lem:cc-wu-bound} we know that in local coordinates$\mathsf{\ x}%
,\mathsf{y}$ the unstable fiber of $q(x)$ is a horizontal disc in $B$. This
disc is a graph of a function $w^{u}:B_{u}\rightarrow B_{c}$ and for any
$\mathsf{x}_{1},\mathsf{x}_{2}\in B_{u}$ such that $\mathsf{x}_{1}%
\neq\mathsf{x}_{2}$%
\[
\left(  \mathsf{x}_{1},w^{u}(\mathsf{x}_{1})\right)  -\left(  \mathsf{x}%
_{2},w^{u}(\mathsf{x}_{2})\right)  \in Q^{+}(0)\subset V.
\]
The disc also passes through the set $B_{E}$ (see Figure \ref{fig:Loc-bound}).

In the statement of our lemma we implicitly assume that $\mathcal{G}\left(
\mathsf{x},\mathsf{y}\right)  $ is well defined for $\left(  \mathsf{x}%
,\mathsf{y}\right)  \in B_{E}.$ This means that
\begin{equation}
W^{u}(L(x))\cap\Sigma\cap\mathcal{G}\left(  B_{E}\right)  =\{\mathcal{G}%
\left(  \mathsf{x},w^{u}(\mathsf{x})\right)  |\mathsf{x}\in\lbrack
\mathsf{x}^{l},\mathsf{x}^{r}]\}. \label{eq:temp-wu-on-sec}%
\end{equation}

Let us introduce a notation%
\begin{align*}
w_{\Sigma}^{u}  &  :[\mathsf{x}^{l},\mathsf{x}^{r}]\rightarrow\mathbb{R}%
^{2},\\
w_{\Sigma}^{u}(\mathsf{x})  &  =\pi_{x,p_{x}}\mathcal{G}\left(  \mathsf{x}%
,w^{u}(\mathsf{x})\right)  .
\end{align*}
By (\ref{eq:py-well-def}) and (\ref{eq:temp-wu-on-sec}) the curve $w_{\Sigma
}^{u}\left(  \mathsf{x}\right)  $ parametrizes a fragment of the intersection
of $W^{u}(L(x))$ with $\Sigma$. By assumption (\ref{eq:exit-set-ass})
\begin{align*}
\pi_{p_{x}}w_{\Sigma}^{u}\left(  \mathsf{x}^{l}\right)   &  =\pi_{p_{x}%
}\mathcal{G}\left(  \mathsf{x}^{l},w^{u}(\mathsf{x}^{l})\right)  \in\pi
_{p_{x}}\mathcal{G}\left(  \left\{  \mathsf{x}^{l}\right\}  \times
B_{c}\right)  =\pi_{p_{x}}\mathcal{G}\left(  B_{E}^{l}\right)  <0,\\
\pi_{p_{x}}w_{\Sigma}^{u}\left(  \mathsf{x}^{r}\right)   &  =\pi_{p_{x}%
}\mathcal{G}\left(  \mathsf{x}^{r},w^{u}(\mathsf{x}^{r})\right)  \in\pi
_{p_{x}}\mathcal{G}\left(  \left\{  \mathsf{x}^{r}\right\}  \times
B_{c}\right)  =\pi_{p_{x}}\mathcal{G}\left(  B_{E}^{r}\right)  >0,
\end{align*}
hence we have an $\mathsf{x}^{\ast}\in\left(  \mathsf{x}^{l},\mathsf{x}%
^{r}\right)  $ such that%
\[
\pi_{p_{x}}w_{\Sigma}^{u}\left(  \mathsf{x}^{\ast}\right)  =0.
\]

The unstable manifold $W^{s}(L(x))$ is $S$-symmetric to $W^{u}(L(x)).$ This
means that a fragment of intersection of $W^{s}(L(x))$ with $\Sigma$ is
parameterized by%
\begin{align}
w_{\Sigma}^{s}  &  :[\mathsf{x}^{l},\mathsf{x}^{r}]\rightarrow\mathbb{R}%
^{2},\nonumber\\
w_{\Sigma}^{s}(\mathsf{x})  &  =\left(  \pi_{x}w_{\Sigma}^{u}\left(
\mathsf{x}\right)  ,-\pi_{p_{x}}w_{\Sigma}^{u}\left(  \mathsf{x}\right)
\right)  . \label{eq:ws-sig-def}%
\end{align}
Since $w_{\Sigma}^{u}\left(  \mathsf{x}^{\ast}\right)  =w_{\Sigma}^{s}\left(
\mathsf{x}^{\ast}\right)  $ manifolds $W^{u}(L(x))$ and $W^{s}(L(x))$
intersect at 
\[
q^{\ast}=\mathcal{G}\left(  \mathsf{x}^{\ast},w^{u}(\mathsf{x}^{\ast})\right).
\]

Now we turn to proving transversality of the intersection at $q^{\ast}$. By
(\ref{eq:py-well-def}), around $q^{\ast}$ the manifold $M(H(L(x)))$ is
parameterized by $x,y,p_{x}.$ Therefore in the proof of transversality we
restrict to these coordinates. Since $\mathcal{G}$ is well defined,
$W^{u}(L(x))$ must transversally cross $\{y=0\}.$ By symmetry so does
$W^{s}(L(x))$. We therefore only need to prove that $w_{\Sigma}^{u}\left(
\mathsf{x}\right)  $ and $w_{\Sigma}^{s}\left(  \mathsf{x}\right)  $ intersect
transversally in $\mathbb{R}^{2}.$

Let $\mathsf{x}^{+}\in(\mathsf{x}^{\ast},\mathsf{x}^{r}],$ $\gamma=1/\left(
\mathsf{x}^{+}-\mathsf{x}^{\ast}\right)  $ and
\[
v=\gamma\left(  \left(  \mathsf{x}^{+},w^{u}(\mathsf{x}^{+})\right)  -\left(
\mathsf{x}^{\ast},w^{u}(\mathsf{x}^{\ast})\right)  \right)  \in V^{+}.
\]
By the mean value theorem%
\[
w_{\Sigma}^{u}\left(  \mathsf{x}^{+}\right)  -w_{\Sigma}^{u}\left(
\mathsf{x}^{\ast}\right)  \in\pi_{x,p_{x}}\frac{1}{\gamma}\left[
D\mathcal{G}(B_{E})\right]  v.
\]
By (\ref{eq:slope-cond1}) this implies that%
\begin{equation}
\pi_{x}(w_{\Sigma}^{u}\left(  \mathsf{x}^{+}\right)  -w_{\Sigma}^{u}\left(
\mathsf{x}^{\ast}\right)  )>0,\qquad\pi_{p_{x}}(w_{\Sigma}^{u}\left(
\mathsf{x}^{+}\right)  -w_{\Sigma}^{u}\left(  \mathsf{x}^{\ast}\right)  )>0.
\label{eq:trans1}%
\end{equation}
By mirror arguments, for $\mathsf{x}^{-}\in\lbrack\mathsf{x}^{l}%
,\mathsf{x}^{\ast})$
\begin{equation}
\pi_{x}(w_{\Sigma}^{u}\left(  \mathsf{x}^{-}\right)  -w_{\Sigma}^{u}\left(
\mathsf{x}^{\ast}\right)  )<0,\qquad\pi_{p_{x}}(w_{\Sigma}^{u}\left(
\mathsf{x}^{-}\right)  -w_{\Sigma}^{u}\left(  \mathsf{x}^{\ast}\right)  )<0.
\label{eq:trans2}%
\end{equation}

From (\ref{eq:trans1}), (\ref{eq:trans2}) and (\ref{eq:ws-sig-def}) we see
that $w_{\Sigma}^{u}\left(  \mathsf{x}\right)  $ and $w_{\Sigma}^{s}\left(
\mathsf{x}\right)  $ intersect transversally at $w_{\Sigma}^{u}\left(
\mathsf{x}^{\ast}\right)  =w_{\Sigma}^{s}\left(  \mathsf{x}^{\ast}\right)  ,$
which concludes our proof.
\end{proof}

\begin{remark}
\label{rem:slope}From proof of Lemma \ref{lem:transversality} follows that we
have the following estimate on the slope of the curves $w_{\Sigma}^{u}\left(
\mathsf{x}\right)  $%
\[
\mathbf{a}=\left[  \frac{\pi_{p_{x}}D\mathcal{G}(B_{E})V^{+}}{\pi
_{x}D\mathcal{G}(B_{E})V^{+}}\right]  \cup\left[  \frac{\pi_{p_{x}%
}D\mathcal{G}(B_{E})V^{-}}{\pi_{x}D\mathcal{G}(B_{E})V^{-}}\right]  .
\]

By $S$-symmetry of $W^{u}(L(x))$ and $W^{s}(L(x))$ the slope of $w_{\Sigma
}^{s}\left(  \mathsf{x}\right)  $ is in $-\mathbf{a.}$

Once we verify (\ref{eq:exit-set-ass}) then by checking that $\mathbf{a}>0$ we
know that assumption (\ref{eq:slope-cond1}) needs to hold.
\end{remark}

\subsection{Proof of Theorem \ref{th:main}\label{sec:proof-of-main-th}}

In this section we write the computer assisted rigorous bounds, which we obtain using the method from sections \ref{sec:fiber-enclosure},
\ref{sec:transversality}. As a result we obtain rigorous bounds for the
position of fibers of $W^{u}(L(x))$ and for transversal intersection of
$W^{u}(L(x))$ with $W^{s}(L(x)).$ By this we obtain the proof of Theorem
\ref{th:main}.

\subsubsection{Bounds for Unstable Fibers\label{sec:Rig-fibers}}

We start by writing out our changes of coordinates needed for application of
Lemma \ref{sec:transversality} to the map (\ref{eq:F-loc}) from section
\ref{sec:fiber-enclosure}.

We first choose the point $q_{0}=(x^{0},0,0,p_{y}^{0})$ with $x^{0},p_{y}^{0}$
given in (\ref{eq:x0-I}) and (\ref{eq:p0y}) respectively, i.e.%
\begin{align*}
x^{0}  &  =-0.9510055339445208,\\
p_{y}^{0}  &  =-0.8368041796469730.
\end{align*}
We choose a matrix $C$ as%
\[
C=\left(
\begin{array}
[c]{llll}%
0.197841 & -0.197841 & 0 & 0.221884\\
-0.221884 & -0.221884 & 0.773671 & 0\\
1 & 1 & -1 & 0\\
-0.255717 & 0.255717 & 0 & -1
\end{array}
\right)
\]
We then choose four polynomials
\begin{align*}
K_{0}(\mathsf{x})=  &  0.1\mathsf{x}-0.0621591\mathsf{x}^{2}%
+0.0375888\mathsf{x}^{3}-0.0200645\mathsf{x}^{4}\\
K_{1}(\mathsf{x})=  &  0.000533561\mathsf{x}^{2}-0.00723085\mathsf{x}%
^{3}+0.00827176\mathsf{x}^{4}\\
K_{2}(\mathsf{x})=  &  -0.0151949\mathsf{x}^{2}+0.009304476\mathsf{x}%
^{3}-0.00427633\mathsf{x}^{4}\\
K_{3}(\mathsf{x})=  &  0.0269670\mathsf{x}^{2}-0.0275820{x}^{3}%
+0.0203022\mathsf{x}^{4}%
\end{align*}
which define the nonlinear change of coordinates $\psi$ (see
(\ref{eq:psi-form})). All of the above choices are dictated by (non-rigorous)
numerical investigation. Above choice ensures that $C\psi(\mathsf{x},0)+q_{0}$
gives a decent approximation of the position of the unstable fibers of $q(x)$
for $x\in I$ for $I$ given in (\ref{eq:x0-I}).

Now our computations start. We first compute the interval enclosure
$\mathbf{T}$ such that $\tau(q(x))\in\mathbf{T}$ for all $x\in I$. The
obtained result is
\[
\mathbf{T}=\mathtt{[3.058882598,3.058883224]}.
\]
\qquad\qquad

Next we compute a set $B_{0}$ such that (see Figure \ref{fig:Loc-bound})
\[
\psi^{-1}(C^{-1}(q(x)-q^{0}))\subset B_{0}.
\]
Such set can be obtained using a technical Lemma \ref{lem:Newton-preimage}
included in the Appendix. We thus obtain%
\[
B_{0}=\left(
\begin{array}
[c]{c}%
\mathtt{[-7.91575e-12,7.91575e-12]}\\
\mathtt{[-7.91575e-12,7.91575e-12]}\\
\mathtt{[-9.29424e-19,9.29424e-19]}\\
\mathtt{[-4.50827e-08,4.50827e-08]}%
\end{array}
\right)  .
\]

\begin{remark}
Note that the set is flat along the third and stretched along the last
coordinate. This is because we set up $C$ and $\psi$ so that the third
coordinate is associated with the section $\{y=0\}$ (on which lie $q(x)$) and
that the last coordinate is associated with the direction of the curve
$q(x)=(x,0,0,\kappa(x))$.
\end{remark}

We now choose the size of our investigated set $B$ in local coordinates and
choose the parameters for our cones (see Figure \ref{fig:Loc-bound}). We take
\[
\alpha=2.56\cdot10^{-6},
\]
and consider only one branch of the unstable manifold considering
\begin{equation}
B=\bigcup_{v\in B_{0}}Q^{+}(v)\cap\{\mathsf{x}\in\lbrack\underline{\mathsf{x}%
},\overline{\mathsf{x}}]\} \label{eq:B-set-num}%
\end{equation}
with
\[
\underline{\mathsf{x}}=-1\cdot10^{-11},\qquad\overline{\mathsf{x}}%
=4.5\cdot10^{-6}.
\]

The choice of $\overline{\mathsf{x}}$ is dictated by the size of the fiber we
later need to consider to prove intersections of stable/unstable manifolds.

To compute a rigorous enclosure of $[DF(B)]$ using (\ref{eq:DFloc-form}), we
subdivide $B$ into $N=1200$ parts $B_{i}$ along the $\mathsf{x}$ coordinate
\[
B=\bigcup_{i=1}^{N}B_{i}.
\]
Using Lemma \ref{lem:F-image} to obtain enclosures of $F(B_{i}),$ combined
with (\ref{eq:DFloc-form}) we compute estimates for $[DF(B_{i})].$ Combining
the estimates $[DF(B_{i})]$ we obtain the following global estimate for
$[DF(B)]$ (the result is displayed with very rough accuracy, ensuring true
enclosure in rounding)%
\begin{multline*}
\lbrack DF(B)]=\\
\left(
\begin{array}
[c]{llll}%
\mathtt{[1465.6,1466.5]} & \mathtt{[-0.353,0.369]} & \mathtt{[-0.285,0.283]} &
\mathtt{[-0.300,0.333]}\\
\mathtt{[-0.361,0.360]} & \mathtt{[-0.360,0.361]} & \mathtt{[-0.290,0.277]} &
\mathtt{[-0.319,0.304]}\\
\mathtt{[-0.138,0.140]} & \mathtt{[-0.139,0.139]} & \mathtt{[0.896,1.120]} &
\mathtt{[0.458,0.700]}\\
\mathtt{[-0.201,0.202]} & \mathtt{[-0.202,0.202]} & \mathtt{[-0.171,0.149]} &
\mathtt{[0.823,1.172]}%
\end{array}
\right)  .
\end{multline*}

Finally, using $[DF(B)]$ and Lemmas \ref{lem:cc2-expansion-condition},
\ref{lem:cc1-matrix} we verify assumptions of Lemma \ref{lem:cc-wu-bound}. We
thus obtain rigorous bounds for the position of the fibers. The computation of
the enclosure of the fibers took $18$ minutes on a standard laptop.

We plot the obtained bounds on fibers transported to the original coordinates of the system $x,y,p_x,p_y$ in Figures \ref{fig:fibers-x-y},
\ref{fig:fibers-x-px}, \ref{fig:fibers-x-py}. On the plots we present rigorous
enclosures of three fibers starting from $q(x)$ with $x$ on the edges and the
middle of interval $I$ (with $I$ chosen in (\ref{eq:x0-I})). This gives us an
overview of the size of our fiber enclosures (left hand side of Figures
\ref{fig:fibers-x-y}, \ref{fig:fibers-x-px}, \ref{fig:fibers-x-py}). We can
see that close to the set which contains $\{q(x)=(x,0,0,\kappa(x))|x\in I\}$,
which is depicted in green, the estimates on the fibers is sharp (right hand
plots in Figures \ref{fig:fibers-x-y}, \ref{fig:fibers-x-px},
\ref{fig:fibers-x-py}). We can see that our three considered fiber enclosures
are very close to each other, but are still separated, which is visible after
closeup on the left hand side plot in Figure \ref{fig:fibers-x-py}.

\begin{figure}[h]
\begin{center}
\includegraphics[height=4.1cm]{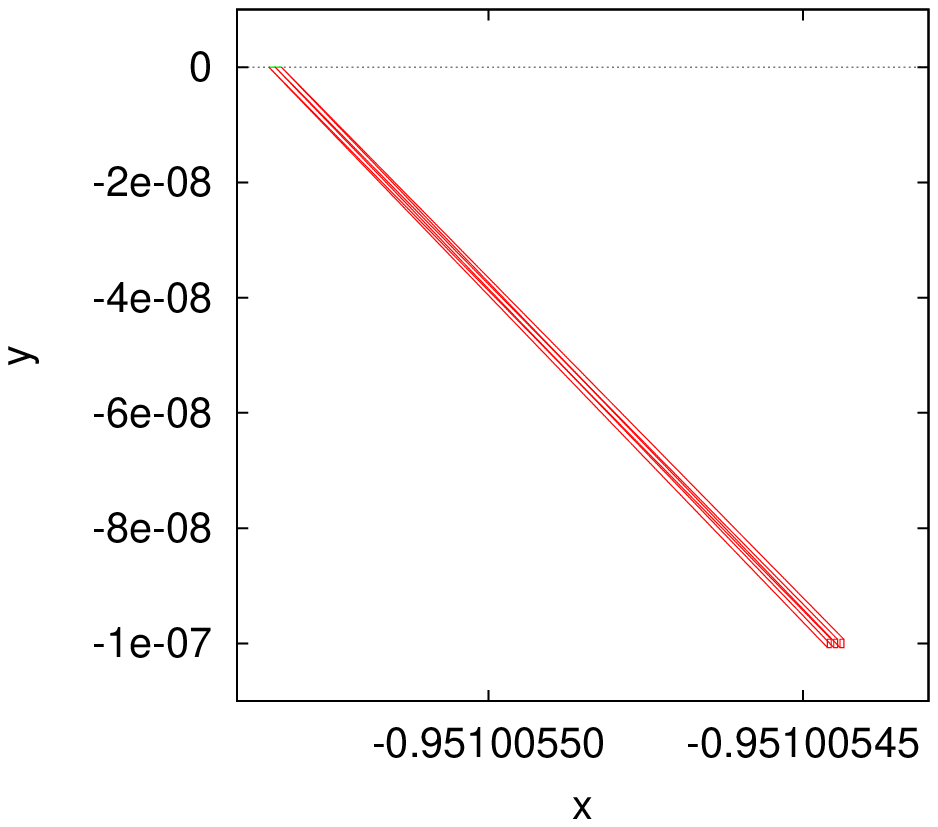}
\includegraphics[height=4.1cm]{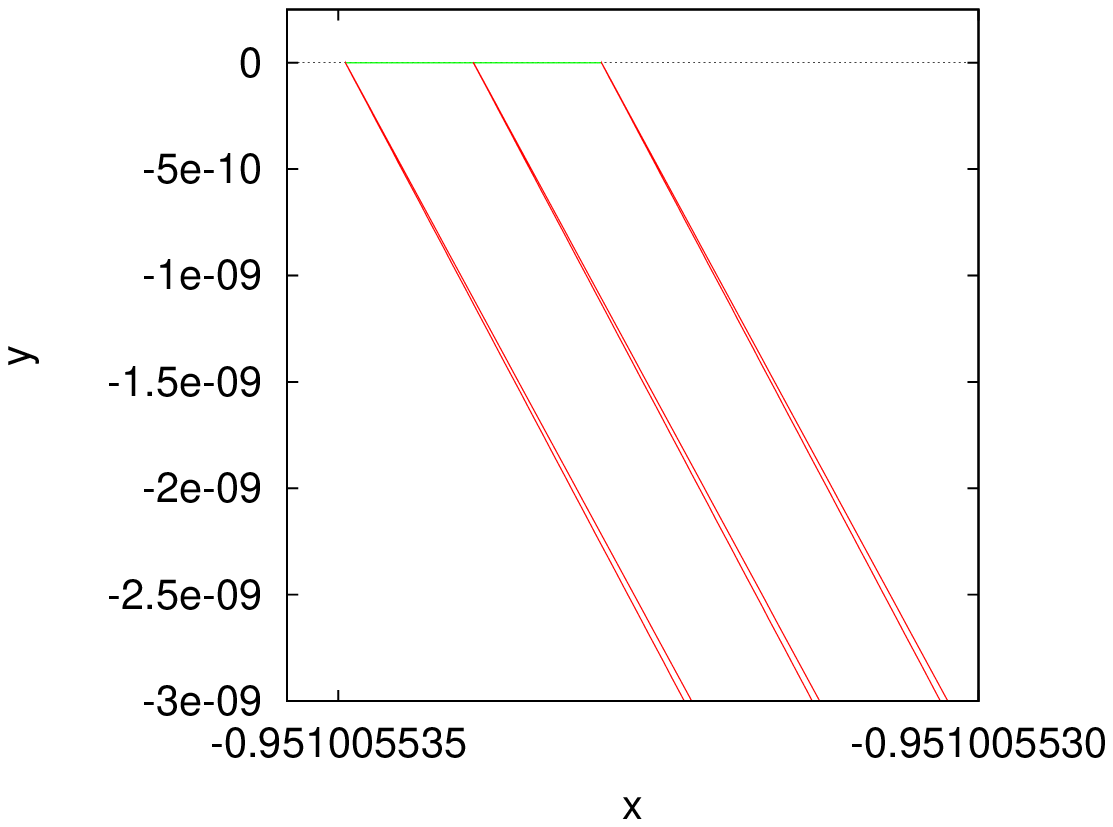}
\end{center}
\caption{Projections of fiber enclosures onto $x,y$ coordinates. }%
\label{fig:fibers-x-y}%
\end{figure}\begin{figure}[h]
\begin{center}
\includegraphics[height=4.1cm]{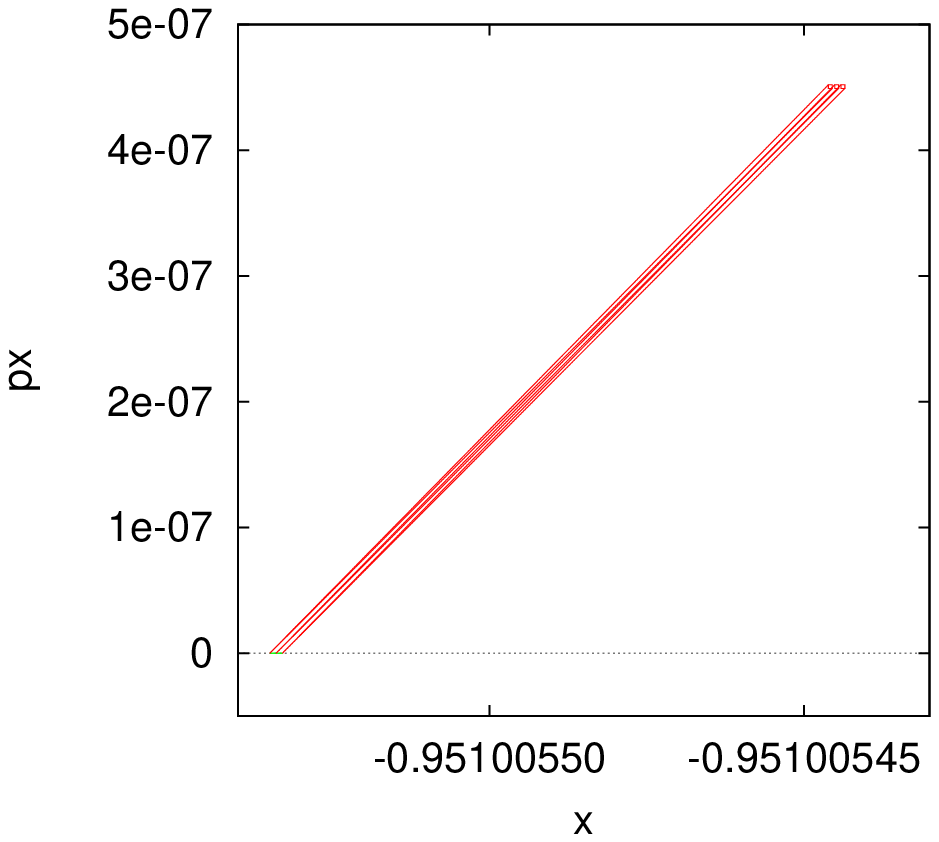}
\includegraphics[height=4.1cm]{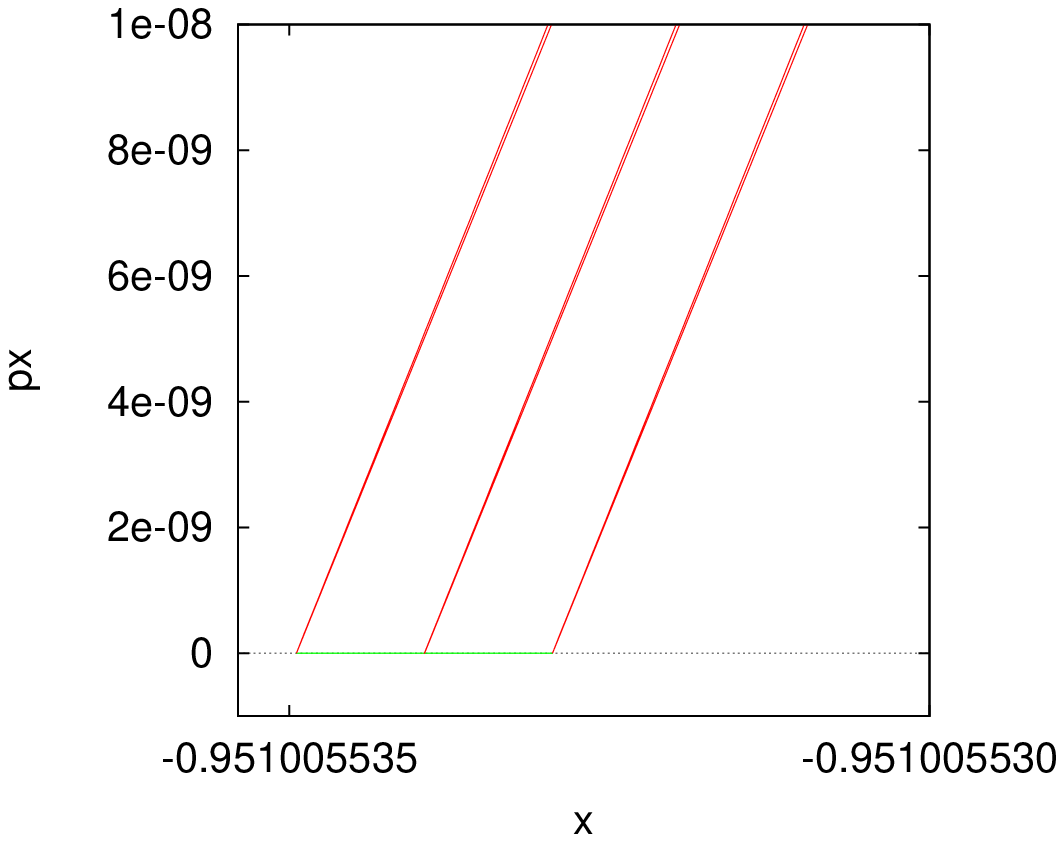}
\end{center}
\caption{Projections of fiber enclosures onto $x,p_{x}$ coordinates. }%
\label{fig:fibers-x-px}%
\end{figure}\begin{figure}[h]
\begin{center}
\includegraphics[height=4.1cm]{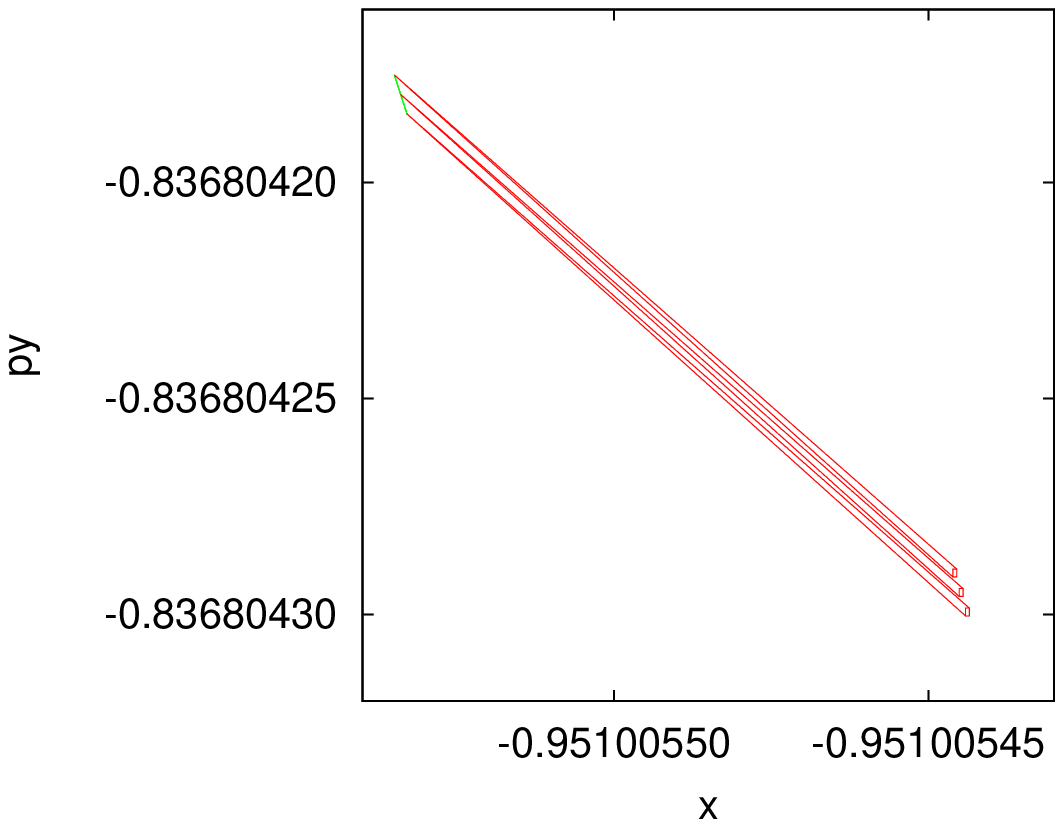}
\includegraphics[height=4.1cm]{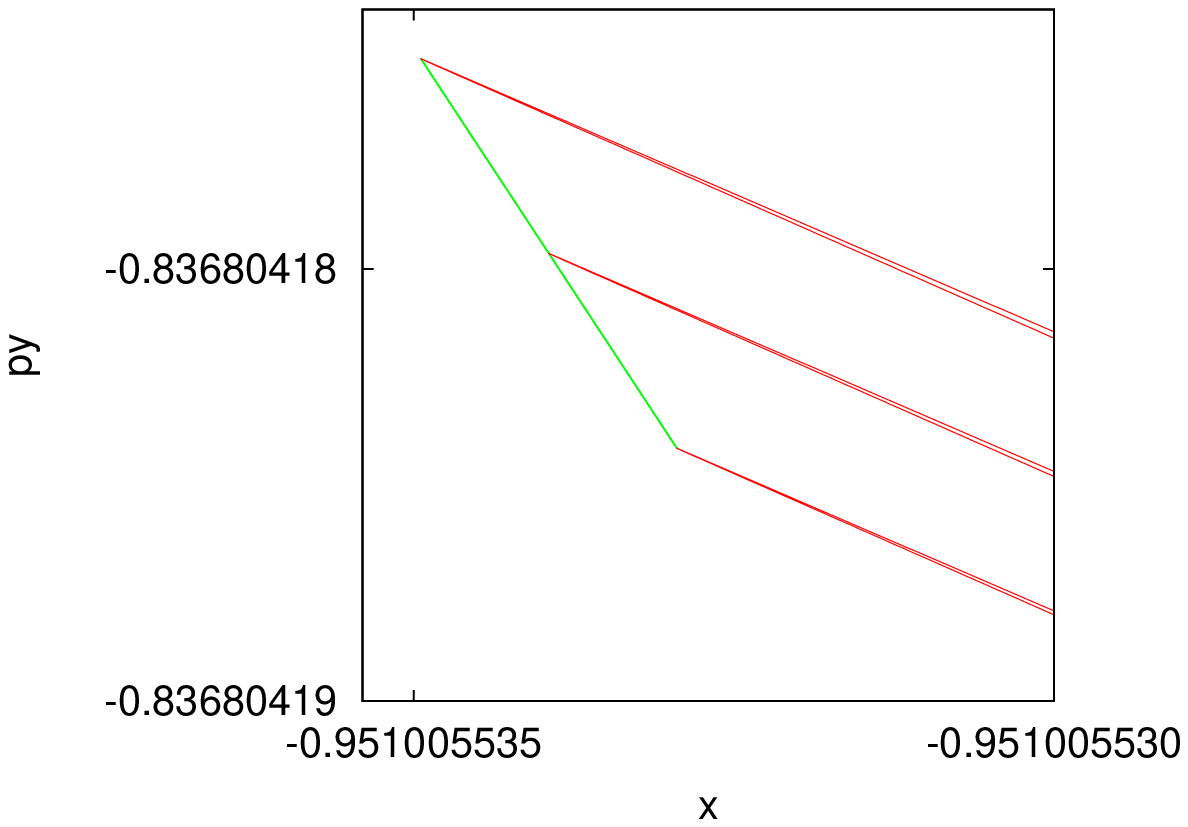}
\end{center}
\caption{Projections of fiber enclosures onto $x,p_{y}$ coordinates. }%
\label{fig:fibers-x-py}%
\end{figure}

\begin{remark}
The range of obtained fibers is small. It is possible to reach
somewhat further from $q(x)$, but this significantly increases the time of
computation, since further subdivision of the set is required.
\end{remark}

\begin{remark}
By using linearization only we have not been able to obtain accurate enough
enclosure of the fibers to handle the proof of transversal intersection of
manifolds which follows in section \ref{sec:intersection-numerics}. Thus the
use of higher order change of variables seems to be needed.
\end{remark}

\subsubsection{Bounds for Intersections of
Manifolds\label{sec:intersection-numerics}}

In this section we present rigorous-computer-assisted results in which we
verify assumptions of Lemma \ref{lem:transversality} and thus conclude the
proof of Theorem \ref{th:main}. For each $x\in I$ there are four points of intersection of $W^{u}(L(x))$ and
$W^{s}(L(x))$ on $\{y=0\}$. They can be seen on the left hand plot in Figure
\ref{fig:intersection}. We consider only the point which is furthermost to the right.

We define the set $B_{E}=\left[  \mathsf{x}^{l},\mathsf{x}^{r}\right]  \times
B_{c}$ with $\mathsf{x}^{l},\mathsf{x}^{r}$ chosen as%
\[
\mathsf{x}^{m}=4.461867506615821\cdot10^{-6},
\]%
\[
\mathsf{x}^{l}=\mathsf{x}^{m}-10^{-11},\qquad\qquad\mathsf{x}^{r}%
=\mathsf{x}^{m}+10^{-11}.
\]

We verify that assumption (\ref{eq:exit-set-ass}) of Lemma
\ref{lem:transversality} holds by computing $\mathcal{G}(B_{E}^{l})$ and
$\mathcal{G}(B_{E}^{r}).$ We plot the obtained bounds in red on the right hand
side plot of Figure \ref{fig:intersection}.

Next, using Remark \ref{rem:slope} we compute
\begin{equation}
\mathbf{a}=\mathtt{[1.7695,1.7725],} \label{eq:a-bound-num}%
\end{equation}
hence assumption (\ref{eq:slope-cond1}) of Lemma \ref{lem:transversality}
holds. Applying Lemma \ref{lem:transversality} concludes the proof of Theorem
\ref{th:main}.

\begin{figure}[h]
\begin{center}
\includegraphics[height=4.1cm]{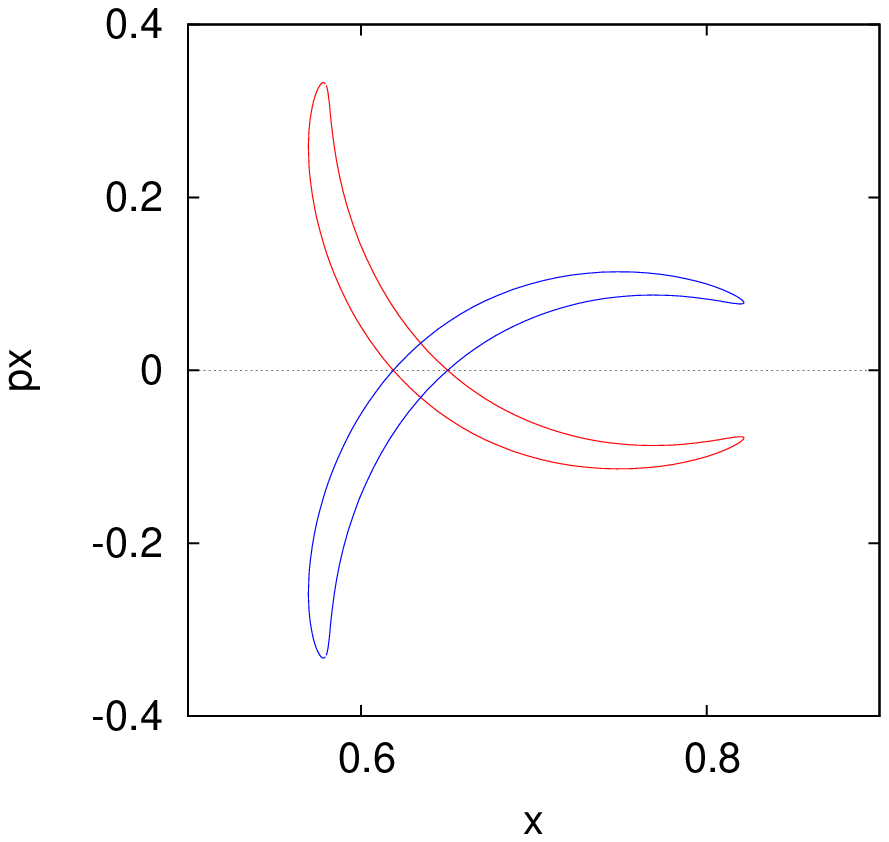}
\includegraphics[height=4.1cm]{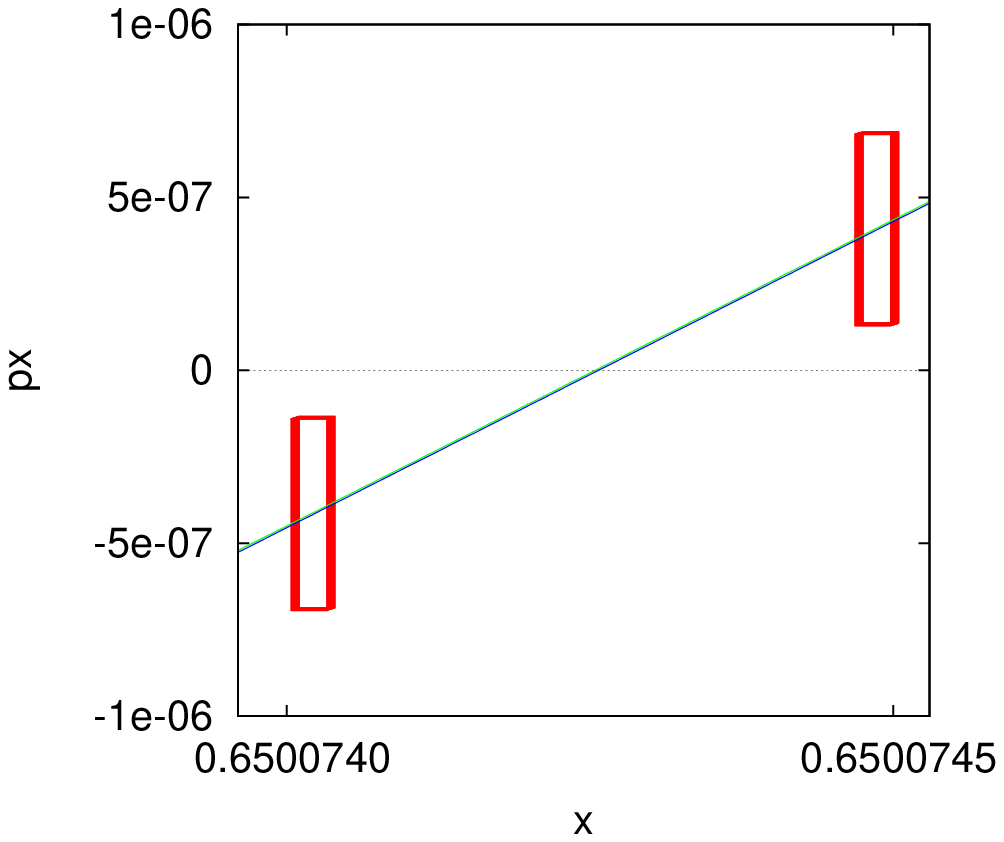}
\end{center}
\caption{Left: Numerical sketch of $W^{s}(L(x))$ (in red) and $W^{u}(L(x))$
(in blue) intersected with $\{y=0\}$. Right: we have proved that
$\{\pi_{x,p_{x}}(W^{u}(L(x))\cap\{y=0\})|x\in I\}$ consists of curves which
pass through two red boxes. We have also proved that their slope is between
$\mathtt{[1.7695,1.7725]}$. The blue/green line is a non-rigorous plot of the
curves.}%
\label{fig:intersection}%
\end{figure}

We needed to subdivide $B_{E}$ into $600$ parts to compute $\left[
D\mathcal{G}(B_{E})V^{+}\right]  $ together with $\left[  D\mathcal{G}(B_{E}%
)V^{-}\right]  $ with sufficient accuracy to obtain (\ref{eq:a-bound-num}).
Verification of assumptions of Lemma \ref{lem:transversality} took $24$
minutes on a standard laptop.

\begin{remark}
From $\mathbf{a}$ and by $S$-symmetry of manifolds $W^{u}(L(x))$ and
$W^{s}(L(x))$, we obtain an estimate $\mathtt{[58.8637}^{\circ}%
\mathtt{,58.9439}^{\circ}\mathtt{]}$ on the angle of intersection of the
curves on the $x,p_{x}$ plane.
\end{remark}

\begin{remark}
At one go we obtain an estimate for a whole family of curves on
\[
\pi_{x,p_{x}}\left(  W^{u}(L(x))\cap\{y=0\}\right)  \quad\text{for }x\in I.
\]
In reality these curves are very close to each other (at furthest distance
along $x$ of about $2.65\cdot10^{-9}$). We plotted (using non-rigorous
computations) two curves which are furthest from each other on the right hand
side plot of Figure \ref{fig:intersection}. One is in green and the other in
blue. They are visible only after a large magnification, and on a paper
printout will merge together. This means that our estimate on the position of
the curves is somewhat rough in comparison to non-rigorous numerical simulation.
\end{remark}

\section{Closing Remarks and Future Work\label{sec:closing-remarks}}
In this paper we have presented a method for proving existence of families of Lyapunov orbits in the planar restricted circular three body problem. The method gives explicit bounds on a curve of initial points, which can continue up to half the distance from $L_2$ to the smaller primary in the Jupiter-Sun system.

We also presented a method of proving transversal intersections of invariant manifolds associated with Lyapunov orbits. The method gives explicit bounds on where the intersection takes place. It has been applied to Lyapunov orbits with energy of the comet Oterma in the Jupiter-Sun system.

In this paper we have focussed on detection of homoclinic intersections. Using identical tools one could also prove heteroclinic intersections of manifolds in the spirit of the work of Wilczak and Zgliczy\'nski \cite{WZ1,WZ2}.

Due to the fact that the presented method gives explicit estimates on the position of investigated manifolds, it is our hope to later apply it to the study of diffusion. 
Here is an outline of future scheme that could be followed to prove diffusion.
The family of Lyapunov orbits is normally hyperbolic, hence survives time periodic perturbations. In non-autonomous setting the system no longer preserves energy, which allows for diffusion between orbits of different energies. Such mechanism has been investigated in \cite{CZ} for the planar restricted elliptic three body problem, for the system with special restriction on parameters. The discussed diffusion follows from the geometric method of Delshams, de la Llave and Seara \cite{DLS3, DLS2, DLS1} and requires computation of Melnikov type integrals along homoclinic orbits of the PRC3BP. Since our method allows for precise and rigorous estimates for such orbits, it is our hope that such integrals could be computed using rigorous-computer assisted techniques. This combined with topological methods \cite{CR, CZ2} for detection of normally hyperbolic manifolds could give first rigorous results for diffusion in the three body problem with real life parameters. From this perspective, the results of this paper are a first step in a larger scheme for investigation of real life systems.

\section{Acknowledgements\label{sec:ackn}}
The author would like to thank Rafael de la Llave for discussions and remarks regarding implementation of the parameterization method. Special thanks go also to Daniel Wilczak for discussions on rigorous-computer-assisted computations using the CAPD library (http://capd.ii.uj.edu.pl).

%\usepackage{pdfsync}
%\usepackage{hyperref}%

%TCIDATA{OutputFilter=latex2.dll}
%TCIDATA{Version=5.00.0.2606}
%TCIDATA{LaTeXparent=0,0,online-edit.tex}

\section{Appendix\label{sec:appendix}}

\subsection{Verification of Cone Conditions}

\begin{lemma}
\label{lem:cc2-expansion-condition}Let $\mathbf{A}$ be an interval matrix of
the form%
\[
\mathbf{A}=\left(
\begin{array}
[c]{cc}%
\mathbf{a}_{11} & \mathbf{\varepsilon}^{T}\\
\mathbf{B} & \mathbf{C}%
\end{array}
\right)
\]
where $\mathbf{a}_{11}=\left[  \underline{a}_{11},\overline{a}_{11}\right]  $
with $\underline{a}_{11}>0$. If for any $\varepsilon\in\mathbf{\varepsilon}$,
$\left\Vert \varepsilon\right\Vert \leq\epsilon$ holds%
\begin{equation}
\frac{\underline{a}_{11}-\epsilon\sqrt{\alpha}}{\sqrt{1+\alpha}}>m,
\label{eq:m-ineq}%
\end{equation}
then for $v=(\mathsf{x},\mathsf{y})$ such that $Q(v)=\alpha\mathsf{x}%
^{2}-\left\Vert \mathsf{y}\right\Vert ^{2}\geq0$ and any $A\in\mathbf{A}$ we
have $\left\Vert Av\right\Vert >m\left\Vert v\right\Vert .$
\end{lemma}

\begin{proof}
For $v=(\mathsf{x},\mathsf{y})$ satisfying $Q(v)\geq0,$ we have $\left\Vert
\mathsf{x}\right\Vert ^{2}+\left\Vert \mathsf{y}\right\Vert ^{2}\leq\left\Vert
\mathsf{x}\right\Vert ^{2}(1+\alpha).$ Using (\ref{eq:m-ineq}) this gives the
following estimate
\[
\left\Vert Av\right\Vert \geq\underline{a}_{11}\left\Vert \mathsf{x}%
\right\Vert -\epsilon\left\Vert \mathsf{y}\right\Vert \geq\left(
\underline{a}_{11}-\epsilon\sqrt{\alpha}\right)  \left\Vert \mathsf{x}%
\right\Vert >m\sqrt{\left\Vert \mathsf{x}\right\Vert ^{2}+\left\Vert
\mathsf{y}\right\Vert ^{2}}=m\left\Vert v\right\Vert .
\]

\end{proof}

\begin{lemma}
\label{lem:cc1-matrix}Let $Q\left(  v\right)  =Q(\mathsf{x},\mathsf{y}%
)=\alpha\mathsf{x}^{2}-\left\Vert \mathsf{y}\right\Vert ^{2},$ let $C_{Q}$ be
a diagonal matrix such that $v^{T}C_{Q}v=Q(v),$ and let $\mathbf{A}=[DF(Q^+(v^*))].$
Assume that $\mathbf{D=\mathbf{A}^{T}}C_{Q}\mathbf{\mathbf{A}}$ is an interval
matrix of the form
\[
\mathbf{D}=\left(
\begin{array}
[c]{cc}%
\mathbf{d}_{11} & \boldsymbol{\varepsilon}^{T}\\
\boldsymbol{\varepsilon} & \mathbf{B}%
\end{array}
\right)  .
\]
Assume that $\mathbf{d}_{11}=\left[  \underline{d}_{11},\overline{d}%
_{11}\right]  $ with $\underline{d}_{11}>0$ and that for some $M>0,$ for any
symmetric matrix $B\in\mathbf{B}$
\begin{equation}
\inf\left\{  \lambda|\lambda\in\mathrm{spec}\left(  B\right)  \right\}  >-M.
\label{eq:spect-bound}%
\end{equation}
If for any $\varepsilon\in\boldsymbol{\varepsilon}$ we have $\left\Vert
\varepsilon\right\Vert \leq\epsilon$ and $\underline{d}_{11}-2\epsilon
>M\alpha$, then for any $v_{1},v_{2}\in U,$ $v_{1}\neq v_{2}$ such that
$Q\left(  v_{1}-v_{2}\right)  \geq0$
\[
Q\left(  F(v_{1})-F(v_{2})\right)  >0.
\]

\end{lemma}

\begin{proof}
By (\ref{eq:cc1-interval-matrix}) $Q\left(  F(v_{1})-F(v_{2})\right)  =\left(
v_{1}-v_{2}\right)  ^{T}D\left(  v_{1}-v_{2}\right)  $ for some symmetric
matrix $D\in\mathbf{D}$.

For $v=\left(  \mathsf{x},\mathsf{y}\right)  $ such that $Q(\mathsf{x}%
,\mathsf{y})\geq0$ and for any symmetric $D\in\mathbf{D}$
\[
D=\left(
\begin{array}
[c]{cc}%
d_{11} & \varepsilon^{T}\\
\varepsilon & B
\end{array}
\right)
\]
we compute the following bounds%
\begin{align*}
v^{T}Av  &  =d_{11}\mathsf{x}^{2}+\mathsf{x}\varepsilon^{T}\mathsf{y}%
+\mathsf{y}^{T}\varepsilon\mathsf{x}+\mathsf{y}^{T}B\mathsf{y}\\
&  \geq\underline{d}_{11}\mathsf{x}^{2}-2\epsilon\left\Vert \mathsf{y}%
\right\Vert \left\vert \mathsf{x}\right\vert -M\left\Vert \mathsf{y}%
\right\Vert ^{2}\\
&  \geq\left(  \underline{d}_{11}-2\epsilon\right)  \mathsf{x}^{2}-M\left\Vert
\mathsf{y}\right\Vert ^{2}\\
&  =M\left(  \frac{\underline{d}_{11}-2\epsilon}{M}\mathsf{x}^{2}-\left\Vert
\mathsf{y}\right\Vert ^{2}\right) \\
&  >MQ\left(  \mathsf{x},\mathsf{y}\right) \\
&  >0.
\end{align*}

\end{proof}

\begin{remark}
Assumption (\ref{eq:spect-bound}) is easily verifiable from Gershgorin theorem.
\end{remark}

\subsection{Bounds for the Images in Local Coordinates
\label{sec:Local-map-F(U)}}
Here we give a proof of Lemma \ref{lem:F-image}.
\begin{proof}Inclusion (\ref{eq:F-image}) is equivalent
to showing that for any $\tau\in\mathbf{T,}$ and any $v_{1}\in U_{1}$ there
exists an $v_{2}\mathbf{\ }$in $U_{2}$ such that
\begin{equation}
G(\tau,v_{1},v_{2})=0.\label{eq:G=0}%
\end{equation}
Let us fix a $\tau\in\mathbf{T}$ and $v_{1}\in U_{1}$ and use a notation
$G_{\tau,v_{1}}(v_{2})=G(\tau,v_{1},v_{2})$. Observe that $\left[
DG_{\tau,v_{1}}(U_{2})\right]  \subset\mathbf{A}(U_{2})$ and $\left[
G_{\tau,v_{1}}(v_{0})\right]  \subset\left[  G(\mathbf{T},U_{1},v_{0})\right]
.$ Since from (\ref{eq:Floc-Newton-assumption})
\[
v_{0}-\left[  DG_{\tau,v_{1}}(U_{2})\right]  ^{-1}G_{\tau,v_{1}}(v_{0})\subset
N(\mathbf{T},v_{0},U_{1},U_{2})\subset U_{2},
\]
by the interval Newton method (Theorem \ref{th:interval-Newton}) there exists
a unique $v_{2}=v_{2}\left(  \tau,v_{1}\right)  \in U_{2},$ which satisfies
(\ref{eq:G=0}).
\end{proof}

\begin{remark}
\label{rem:preimage}When applying Lemma \ref{lem:F-image}, due to very strong
hyperbolicity of the map $\Phi_{\tau}$ it pays off to use the mean value
theorem. Taking $U_{1}=v_{1}+B$ we can compute%
\[
N(\mathbf{T},v_{0},U_{1},U_{2})=v_{0}-\left[  \mathbf{A}\left(  U_{2}\right)
\right]  ^{-1}G(\mathbf{T},v_{1},v_{0})-\left[  \left(  \mathbf{A}\left(
U_{2}\right)  ^{-1}\frac{\partial G}{\partial v_{1}}(\mathbf{T},U_{1}%
,v_{0})\right)  B\right]  .
\]
This is a better form since in (below we neglect arguments in order to keep
the formula compact)
\begin{equation}
\mathbf{A}^{-1}\frac{\partial G}{\partial v_{1}}=-D\lambda^{-1}\cdot\left(
\left(  D\psi\right)  ^{-1}\cdot D\tilde{\Phi}_{\tau}\cdot D\psi\right)
,\label{eq:A-implicit}%
\end{equation}
the strong hyperbolic expansion cancels out. This is the main advantage of
Lemma \ref{lem:F-image}.
\end{remark}

Here we give a technical lemma that can be used for computation of
\[
\psi^{-1}(C^{-1}\left(  (x,0,0,\kappa(x))-q^{0}\right)  )\quad\text{for }x\in
I
\]
and $q^{0}=(x^{0},0,0,p_{y}^{0})$. Below, $R$ can be any matrix close to $D\psi^{-1}(0)C^{-1}A$.

\begin{lemma}
\label{lem:Newton-preimage}Let $a\in\mathbb{R}$ and $J_{1}\subset\mathbb{R}$
be from Lemma \ref{lem:Newton-per-orb} and let%
\[
A=\left(
\begin{array}
[c]{llll}%
1 & 0 & 0 & 0\\
0 & 1 & 0 & 0\\
0 & 0 & 1 & 0\\
a & 0 & 0 & 1
\end{array}
\right)  .
\]
Let $\mathbf{B}$ be a set in $\mathbb{R}^{4}$, let $R$ be a $4\times4$ matrix
and let
\[
M:=\left[  A^{-1}PD\psi\left(  R\mathbf{B}\right)  R\right]  ^{-1}\left(
I-x^{0},0,0,J_{1}-p_{y}^{0}\right)  \mathbf{.}%
\]

If%
\begin{equation}
M\subset\mathbf{B} \label{eq:MinR}%
\end{equation}
then $\psi^{-1}(C^{-1}\left(  (x,0,0,\kappa(x))-q^{0}\right)  )\subset
R\mathbf{B}$.
\end{lemma}

\begin{proof}
By Lemma \ref{lem:Newton-per-orb}
\begin{align*}
\left(  x,0,0,\kappa(x)\right)   &  \in(x^{0},0,0,p_{y}^{0})+\left(
I-x^{0},0,0,a(I-x^{0})+J_{1}-p_{y}^{0}\right) \\
&  =q^{0}+A\left(  I-x^{0},0,0,J_{1}-p_{y}^{0}\right)  ,
\end{align*}
hence%
\begin{equation}
(x,0,0,\kappa(x))-q^{0} \in A\left(  I-x^{0},0,0,J_{1}-p_{y}^{0}\right)  .
\label{eq:qx-by-A}%
\end{equation}

Let%
\[
G_{q}(p)=A^{-1}C\psi(Rp)-q
\]
If we can show that for any $q\in\left(  I-x^{0},0,0,J_{1}-p_{y}^{0}\right)  $
there exists a $p\in\mathbf{B}$ such that
\begin{equation}
G_{q}(p)=0 \label{eq:temp-G0}%
\end{equation}
then%
\[
\psi^{-1}\left(  C^{-1}Aq\right)  =Rp,
\]
hence by (\ref{eq:qx-by-A})
\[
\psi^{-1}(C^{-1}\left(  (x,0,0,\kappa(x))-q^{0}\right)  )\subset R\mathbf{B.}%
\]

To show (\ref{eq:temp-G0}) we apply the interval Newton method (Theorem
\ref{th:interval-Newton}). Since $\psi(0)=0$ we can compute%
\begin{align*}
N(0,\mathbf{B})  &  =-\left[  \frac{d}{dp}G_{q}(\mathbf{B})\right]  ^{-1}%
G_{q}(0)\\
&  =-\left[  A^{-1}CD\psi(R\mathbf{B})R\right]  ^{-1}\left(  -q\right) \\
&  \subset M,
\end{align*}
and by (\ref{eq:MinR}) combined with Theorem \ref{th:interval-Newton} obtain
(\ref{eq:temp-G0}), and hence obtain our claim.
\end{proof}

%\usepackage{pdfsync}
%\usepackage{hyperref}%

%TCIDATA{OutputFilter=latex2.dll}
%TCIDATA{Version=5.00.0.2606}
%TCIDATA{LaTeXparent=0,0,online-edit.tex}

\end{document}